\documentclass[runningheads]{llncs}

\usepackage{amsthm}
\usepackage{mathrsfs}

\usepackage[linguistics]{forest}

\usepackage{bussproofs}
\EnableBpAbbreviations
\usepackage{bpextra}

\usepackage{colonequals}

\usepackage{lscape}
\usepackage{pdflscape}

\usepackage{stmaryrd}
\usepackage{mathdots}

\usepackage[geometry]{ifsym}
\usepackage{cmll}

\usepackage{comment}

\usepackage{color}
\usepackage{xspace}

\usepackage{mathtools}
\usepackage{enumerate}

\usepackage{centernot}

\usepackage{hyperref}

\usepackage{cleveref}

\usepackage{multirow}

\usepackage{array}
\newcolumntype{C}[1]{>{\centering\arraybackslash}p{#1}}
\newcolumntype{L}[1]{>{\arraybackslash}p{#1}}

\usepackage{color}
\usepackage{cancel}
\usepackage{verbatim,cmll}
\usepackage{setspace}
\usepackage{lscape}
\usepackage{latexsym,relsize}
\usepackage{multicol}
\usepackage[position=top]{subfig}
\usepackage{tikz}

\usepackage{tabularx}

\usepackage{txfonts}

\usepackage{multicol}
\usetikzlibrary{arrows}
\usetikzlibrary{matrix}
\usetikzlibrary{patterns}
\usetikzlibrary{shapes}
\usetikzlibrary{positioning}

\def\aol{\rule[0.5865ex]{1.38ex}{0.1ex}}

\newcommand{\FH}{\hat{f}}
\newcommand{\FHp}{\hat{f}_{1}}

\newcommand{\FC}{\check{f}}
\newcommand{\GH}{\hat{g}}

\newcommand{\GC}{\check{g}}

\newcommand{\FHS}{\hat{f}^{\,\sharp}}
\newcommand{\FCS}{\check{f}^{\,\sharp}}
\newcommand{\GHF}{\hat{g}^{\,\flat}}
\newcommand{\GCF}{\check{g}^{\,\flat}}

\newcommand{\aatop}{\ensuremath{\top}\xspace}
\newcommand{\abot}{\ensuremath{\bot}\xspace}
\newcommand{\aand}{\ensuremath{\wedge}\xspace}
\newcommand{\aor}{\ensuremath{\vee}\xspace}
\newcommand{\ararr}{\ensuremath{\rightarrow}\xspace}
\newcommand{\alarr}{\ensuremath{\leftarrow}\xspace}
\newcommand{\adrarr}{\ensuremath{\,{>\mkern-7mu\raisebox{-0.065ex}{\rule[0.5865ex]{1.38ex}{0.1ex}}}\,}\xspace}
\newcommand{\adlarr}{\ensuremath{\rotatebox[origin=c]{180}{$\,>\mkern-8mu\raisebox{-0.065ex}{\aol}\,$}}\xspace}
\newcommand{\ATOP}{\hat{\top}}
\newcommand{\ABOT}{\ensuremath{\check{\bot}}\xspace}
\newcommand{\AAND}{\ensuremath{\:\hat{\wedge}\:}\xspace}
\newcommand{\AOR}{\ensuremath{\:\check{\vee}\:}\xspace}
\newcommand{\ARARR}{\ensuremath{\:\check{\rightarrow}\:}\xspace}
\newcommand{\ALARR}{\ensuremath{\:\check{\leftarrow}\:}\xspace}
\newcommand{\ADRARR}{\ensuremath{\hat{{\:{>\mkern-7mu\raisebox{-0.065ex}{\rule[0.5865ex]{1.38ex}{0.1ex}}}\:}}}\xspace}
\newcommand{\ADLARR}{\ensuremath{\:\hat{\rotatebox[origin=c]{180}{$\,>\mkern-8mu\raisebox{-0.065ex}{\aol}\,$}}\:}\xspace}

\renewcommand{\epsilon}{\varepsilon}

\newcommand{\vp}{\overline{p}}

\newcommand{\oz}{\overline{z}}

\newcommand{\bba}{\mathbb{A}}
\newcommand{\bbA}{\mathbb{A}}
\newcommand{\bbL}{\mathbb{L}}

\newcommand{\AATOP}{\hat{\top}}

\newcommand{\marginnote}[1]{\marginpar{\raggedright\tiny{#1}}}

\tikzset{
	treenode/.style = {align=center, inner sep=0pt, text centered},
	Ske/.style = {treenode, ellipse, double, draw=black,
		minimum width=6pt, thick},% arbre rouge noir, noeud noir
	PIA/.style = {treenode, ellipse, black, draw=black,
		minimum width=6pt},% arbre rouge noir, noeud rouge
	Crit/.style = {treenode, rectangle, draw=black,
		minimum width=0.5em, minimum height=0.5em}% arbre rouge noir, nil
}

\theoremstyle{plain}
\newtheorem{thm}{Theorem}[section]

\newtheorem{prop}[thm]{Proposition}

\theoremstyle{definition}
\newtheorem{dfn}{Definition}[section]

\newtheorem{rem}[dfn]{Remark}

\def\fCenter{{\mbox{$\ \vdash\ $}}}

\EnableBpAbbreviations

\newcommand{\fns}{\footnotesize}
\newcommand{\mc}{\multicolumn}

\newcommand{\commment}[1]{}
\numberwithin{equation}{section}

\def\pdra{\mbox{$\,>\mkern-8mu\raisebox{-0.065ex}{\aol}\,$}}
\def\pdla{\mbox{\rotatebox[origin=c]{180}{$\,>\mkern-8mu\raisebox{-0.065ex}{\aol}\,$}}}
\newcommand{\pand}{\wedge}

\newcommand{\por}{\vee}

\newcommand{\pra}{\rightarrow}

\newcommand{\bgamma}{\textcolor{blue}{\gamma}}

\makeatletter
\newcommand{\vast}{\bBigg@{4}}
\newcommand{\Vast}{\bBigg@{6}}
\makeatother

\newcommand{\cceq}{\coloncolonequals}
\newcommand{\ceq}{\colonequals}

\makeatletter
\newcommand{\leqnomode}{\tagsleft@true}
\newcommand{\reqnomode}{\tagsleft@false}
\makeatother

\newcommand{\remove}[1]{}

\forestset{
  <--/.style={ % draw horizontal line to predecessor
    before drawing tree={tikz+={\draw[dotted](!)--(!p);}}},
  -->/.style={ % draw horizontal line to successor
    before drawing tree={tikz+={\draw[ thick,dotted](!)--(!n);}}},
  -->>/.style={ % draw horizontal line to successor cousin
    before drawing tree={tikz+={\draw[dotted](!)--(!>);}}},
}

\title{Refutation calculi for lattice-based logics: from display to tableaux}

\date{}

\begin{document}

\author{Andrea De Domenico\inst{1}\orcidID{0000-0002-8973-7011} \and Giuseppe Greco\inst{1}\orcidID{0000-0002-4845-3821} \and Alessandra Palmigiano\inst{1,2}\orcidID{0000-0001-9656-7527} \and Mario Piazza\inst{3}\orcidID{0000-0002-9545-3912} \and Andrea Sabatini\inst{3}\orcidID{0009-0008-3040-9137}}

\institute{Vrije Universiteit Amsterdam, The Netherlands \and
Department of Mathematics and Applied Mathematics, University of Johannesburg, SA \and
Scuola Normale Superiore di Pisa, Italy
}

\authorrunning{De Domenico, Greco, Palmigiano, Piazza, Sabatini.}

\maketitle

\begin{abstract}
Refutation calculi are formal systems developed to derive the invalid formulas of a given logic. While the notion of refutation calculi has played a key role in the development of tableaux calculi, a refutation approach to display calculi has not yet been attempted. In this paper, we introduce refutation display calculi for basic  LE-logics,  i.e.~those logics canonically associated with basic normal lattice expansions of any signature. 
In particular, we prove soundness and completeness via proof-analysis results on derivable sequents. Finally, we obtain terminating tableaux calculi from these refutation display calculi.

\medskip

\end{abstract}
	
	\section{Introduction}

  Refutation calculi are formal systems which derive the invalid formulas of a given logic \cite{GPS}. Starting with Lukasiewicz’s work \cite{Luk51}, refutation calculi for various logics have been developed using different proof-theoretic formats, including Hilbert-style calculi \cite{Skura92,Goranko94}, Gentzen style sequent calculi \cite{Goranko94,PintoDyck95} and hypersequent calculi \cite{PPT24}. In this paper, we investigate refutation calculi based on display calculi. 
  
  Display calculi,  introduced by Belnap \cite{Belnap}, are deductive systems dealing with sequents $\Pi \vdash \Sigma$ where $\Pi$ and $\Sigma$ are {\em structures}, i.e.~syntactic objects inductively defined from formulas using  so-called {\em structural} connectives. The rules that govern the interaction between these connectives ensure that any substructure of any sequent $\Pi \vdash \Sigma$ can always be displayed, either only in precedent position or only in succedent position. 

  Display calculi have been applied to give cut-free presentations of many non-classical logics, including modal and substructural ones \cite{Wansingdisplaybook,Gore98b}. In this paper, we focus on proper display calculi for the family of basic LE-logics,  i.e.~the logics canonically associated with basic normal lattice expansions of any signature \cite{CoPa11,GMPTZ,CGPT22}. Rather than employing display calculi in the usual role of proof construction, we investigate how they can be adapted for generating refutations -- i.e., derivations of invalid sequents, or {\em antisequents}.

 Specifically, we introduce refutation display calculi $\mathrm{D.LE^{r}}$ for basic normal LE-logics. Starting from initial antisequents, these calculi derive new antisequents through the application of display rules, structural rules, and logical introduction rules. We establish soundness and completeness of each  $\mathrm{D.LE^{r}}$ calculus exploiting proof-analysis results on derivable sequents. Hence, we define {\em inverse} refutation systems for LE-logics by taking the contrapositive forms of the rules of $\mathrm{D.LE^{r}}$ calculi. The resulting systems are terminating tableaux calculi for basic LE-logics: if a sequent is valid, then  some branch exists in its tableau tree such that all the sequents at its terminal node  are valid.

%{\color{blue}Display calculi, originally introduced by Belnap, are known for their structural transparency and uniformity, making them particularly suitable for the proof-theoretic analysis of a wide variety of logical systems. Rather than employing them in the usual role of proof construction, we investigate how display calculi can be systematically adapted for refutation purposes, aimed at establishing non-derivability.

%We develop a method for constructing refutation calculi within the display framework and examine its relation to semantic tableaux, a well-established technique for countermodel generation and satisfiability checking. By analyzing the structural correspondences and divergences between these two approaches, we aim to deepen the understanding of their respective strengths and explore how the modularity of display calculi can be harnessed to simulate and generalize tableau-style reasoning.

%A key outcome of this approach is a new route to proving decidability results. Under specific syntactic and structural constraints, we demonstrate that refutation via display calculi leads to terminating procedures for checking the non-derivability of sequents. We provide soundness and completeness results for the refutation systems, and discuss how their design ensures the admissibility of cut and termination of proof search in relevant fragments. Additionally, we consider the computational complexity of these procedures and outline strategies for their automation, including a discussion of implementation perspectives.}

		\paragraph{\textbf{Structure of the paper.}} 
		Section \ref{sec:preliminaries} contains preliminaries on basic $\mathrm{LE}$-logics and their proper display calculi. In Section \ref{sec:refutational}, we define the refutation display calculi $\mathrm{D.LE^{r}}$. In Section \ref{sec:soundness_completeness}, we prove that $\mathrm{D.LE^{r}}$ calculi are sound and complete, and how consequently basic $\mathrm{LE}$-logics are syntactically decidable. In Section \ref{sec:tableau}, we define tableaux systems for $\mathrm{LE}$-logics based on $\mathrm{D.LE^{r}}$. We sketch directions for future research in Section \ref{sec:conclusions}. 

\section{Preliminaries}
\label{sec:preliminaries}

The present section adapts material from \cite[Section 2]{greco2018algebraic}.

\paragraph{\textbf{Basic normal $\mathrm{LE}$-logics.}}
Our basic language is an unspecified but fixed language $\mathcal{L}_\mathrm{LE}$, to be interpreted over  lattice expansions of compatible similarity type. This setting uniformly accounts for many well known logical systems, such as the basic orthomodular logic \cite{Gold74}, the logic of the non-distributive de Morgan algebras \cite{dislatticebook}, the full Lambek calculus \cite{Lambek61} and the full Lambek-Grishin calculus \cite{Grishin83,Moortgat09}.

We use the following auxiliary definition: an {\em order-type} over $n\in \mathbb{N}$ 
is an $n$-tuple $\epsilon\in \{1, \partial\}^n$. 
For every order-type $\epsilon$, we let $\epsilon^\partial$ denote its {\em opposite} order-type, that is, $\ \epsilon^\partial (i) = 1$ if and only if  $ \epsilon(i)=\partial$ for every $1 \leq i \leq n$. %and $ \epsilon^\partial (i) = \partial$ iff $ \epsilon(i)=1$ for every $1 \leq i \leq n$.
%\marginnote{AP: dobbiamo menzionare solo un iff, l'altro e' conseguenza} %For any lattice $\bba$, we let $\bba^1: = \bba$ and $\bba^\partial$ be the dual lattice, that is, the lattice associated with the converse partial order of $\bba$. For any order-type $\varepsilon$, we let $\bba^\varepsilon: = \Pi_{i = 1}^n \bba^{\varepsilon_i}$.

The language $\mathcal{L}_\mathrm{LE}(\mathcal{F}, \mathcal{G})$ (from now on abbreviated as $\mathcal{L}_\mathrm{LE}$) takes as parameters: a denumerable set of proposition letters $\mathsf{AtProp}$, elements of which are denoted $p,q,r$, possibly with indexes, and disjoint sets of connectives $\mathcal{F}$ and $\mathcal{G}$.\footnote{\label{footnote: f and g operators}
	%It will be clear from the treatment in the present and the following sections that 
%	The connectives in $\mathcal{F}$ (resp.~$\mathcal{G}$) correspond to those referred to as {\em positive} (resp.~{\em negative}) connectives in \cite{ciabattoni2008axioms}. This terminology is not adopted in the present paper to avoid confusion with positive and negative nodes in signed generation trees, defined later in this section.
	%is explained later on in Footnote \ref{footnote: why not adopt terminology of CiGaTe}. 
	The assumption that the sets $\mathcal{F}$ and $\mathcal{G}$ are disjoint does not harm generality. For those connectives whose order-theoretic properties make them belong to both $\mathcal{F}$ and $\mathcal{G}$ (this is the case e.g.~of the Boolean negation $\neg$), we can define two copies $\neg_\mathcal{F}\in\mathcal{F}$ and $\neg_\mathcal{G}\in\mathcal{G}$, and introduce structural rules which encode the fact that these two copies coincide. %Another possibility is to admit a non-empty intersection of the sets $\mathcal{F}$ and $\mathcal{G}$. 
    %Notice that only unary connectives can be both left and right adjoints. %Whenever a connective belongs both to $\mathcal{F}$ and to $\mathcal{G}$, a completely standard solution in the display calculi literature is also available.
    } %(cf.~Remark \ref{rem: overloading the notation} and \ref{rem: both f and g operators}).} 
    Each $f\in \mathcal{F}$ and $g\in \mathcal{G}$ has arity $n_f\in \mathbb{N}$ (resp.~$n_g\in \mathbb{N}$) and is associated with some order-type $\varepsilon_f$ over $n_f$ (resp.~$\varepsilon_g$ over $n_g$). 
%\footnote{
%Unary $f \in \mathcal{F}$ (resp.~$g \in \mathcal{G}$) are sometimes denoted  $\Diamond$ (resp.~$\Box$) if their order-type is 1, and $\lhd$ (resp.~$\rhd$) if their order-type is $\partial$.\footnote{The adjoints of the unary connectives $\Box$, $\Diamond$, $\lhd$ and $\rhd$ are denoted $\Diamondblack$, $\blacksquare$, $\blhd$ and $\brhd$, respectively.} 
The terms (formulas) of $\mathcal{L}_\mathrm{LE}$ are defined recursively as follows:
\[
\varphi \cceq p \mid \bot \mid \top \mid \varphi \wedge \varphi \mid \varphi \vee \varphi \mid f(\varphi_1, \ldots, \varphi_{n_f}) \mid g(\varphi_1, \ldots, \varphi_{n_g})
\] 
where $p \in \mathsf{AtProp}$, $f\in\mathcal{F}$ and $g\in\mathcal{G}$. Terms in $\mathcal{L}_\mathrm{LE}$ are denoted either by $A,B,C$, or by lowercase Greek letters such as $\varphi, \psi, \gamma$. 
In the remainder of the paper, %\marginnote{Apostolos: I replaced the double appearance of ``in the remainder of the paper''; hope it is ok.},
we %will often simplify notation and write e.g.~$f$ for $f^\bbA$, $n$ for $n_f$ and $\varepsilon_i$ for $\varepsilon_{f}(i)$. We also 
extend the $\{1,\partial\}$-notation to the symbols $\vee,\wedge,\bot,\top,\le,\vdash$ by stipulating that the superscript $^1$ denotes the identity map, that $\varphi\vdash^\partial \psi$ stands for $\psi\vdash \varphi$, and defining 
$$
\vee^\partial=\wedge,\qquad \wedge^\partial=\vee,\qquad \bot^\partial=\top,\qquad \top^\partial=\bot,\qquad {\le^\partial}={\ge}.
$$

%Normal LEs constitute the main semantic environment of the present paper. 
Henceforth, %every LE is assumed to be normal, so 
the adjective `normal' referred to LEs will typically be dropped. In what follows, for every $k \in \mathcal{F} \cup \mathcal{G}$ we let $k(\overline{\varphi})[\psi]_i$ indicate that  $\psi$ occurs in the $i$-th coordinate of the vector of arguments $\overline{\varphi}$.
The generic LE-logic is not equivalent to a sentential logic. Hence, the consequence relation of these logics needs to be captured in terms of  sequents, motivating the following definition:
\begin{definition}
	\label{def:LE:logic:general}
	For any language $\mathcal{L}_\mathrm{LE} = \mathcal{L}_\mathrm{LE}(\mathcal{F}, \mathcal{G})$, the {\em basic normal} $\mathcal{L}_\mathrm{LE}$-{\em logic} $\mathbf{L}_\mathrm{LE}$ is a set of sequents $\varphi\vdash\psi$, with $\varphi,\psi\in\mathcal{L}_\mathrm{LE}$, which contains  the following  axioms: %for lattice operations and additional connectives: 
	%		\begin{itemize}
	%			\item Sequents for lattice operations:\footnote{In what follows we will use the turnstile symbol $\vdash$ both as sequent separator and also as the consequence relation of the logic.}
	\begin{center}
\begin{tabular}{c}
$\bot\vdash p, \quad p\vdash p, \quad p\vdash \top,\quad p\vdash p \vee q, \quad q\vdash p\vee q, \quad p\wedge q\vdash p, \quad p\wedge q \vdash q,$ \\
\rule[0mm]{0mm}{6mm}$f(\vp)[q\vee^{\epsilon_f(i)} r]_i \vdash f(\vp)[q]_i \vee f(\vp)[r]_i, \qquad f(\vp)[\bot^{\epsilon_f(i)}]_i \vdash \bot,$\\
\rule[0mm]{0mm}{6mm}$g(\vp)[q]_i \wedge g(\vp)[r]_i \vdash g(\vp)[q\wedge^{\epsilon_g(i)} r]_i, \qquad \top \vdash g(\vp)[\top^{\epsilon_g(i)}]_i,$\\
\end{tabular}
\end{center}
	and is closed under the following inference rules:
	\begin{displaymath}
	\frac{\varphi\vdash \chi\quad \chi\vdash \psi}{\varphi\vdash \psi}
	\qquad
	\frac{\varphi\vdash \psi}{\varphi(\chi/p)\vdash\psi(\chi/p)}
	\qquad
	\frac{\chi\vdash\varphi\quad \chi\vdash\psi}{\chi\vdash \varphi\wedge\psi}
	\qquad
	\frac{\varphi\vdash\chi\quad \psi\vdash\chi}{\varphi\vee\psi\vdash\chi}
	\end{displaymath}
	\begin{center}
    \small{
	$\dfrac{\ \ \ \ \varphi\vdash^{\epsilon_{f}(i)}\psi}{f(p_1,\ldots,\varphi,\ldots,p_n)\vdash f(p_1,\ldots,\psi,\ldots,p_n)} \qquad
	\dfrac{\ \ \ \ \varphi \vdash^{\epsilon_{g}(i)}\psi}{g(p_1,\ldots,\varphi,\ldots,p_n)\vdash g(p_1,\ldots,\psi,\ldots,p_n)}$}
	\end{center}
%	
	%\medskip
%	
	%We let $\mathbf{L}_\mathrm{LE}$ denote the minimal $\mathcal{L}_\mathrm{LE}$-logic. We typically drop reference to the parameters when they are clear from the context. By an {\em $\mathrm{LE}$-logic} we understand any axiomatic extension of $\mathbf{L}_\mathrm{LE}$ in the language $\mathcal{L}_{\mathrm{LE}}$. 
\end{definition}
\iffalse
%For every LE $\bba$, the symbol $\vdash$ is interpreted as the lattice order $\leq$. 
A sequent $\varphi\vdash\psi$ is valid in  an LE $\bba$ if $v(\varphi)\leq v(\psi)$ for every homomorphism $v$ from the $\mathcal{L}_\mathrm{LE}$-algebra of formulas over $\mathsf{AtProp}$ to $\bba$.\footnote{Of course, the restriction of every such homomorphism $v$ to the set of proposition variables is a {\em variable assignment}, and conversely, as is well known,  every variable assignment $v: \mathsf{AtProp} \to \bba$  uniquely extends to a homomorphism from the $\mathcal{L}_\mathrm{LE}$-algebra of formulas over $\mathsf{AtProp}$ to $\bba$. In the remainder of this paper, we will abuse notation and use the same symbol to denote both variable assignments and their homomorphic extensions. Also, sometimes we write $\varphi^{\bba}$ for $v(\varphi)$ when  the interpretation of $v$ is unambiguous.} The notation $\mathbb{LE}\models\varphi\vdash\psi$ indicates that $\varphi\vdash\psi$ is valid in every LE of the appropriate signature. Then, by means of a routine Lindenbaum-Tarski construction, it can be shown that the minimal LE-logic $\mathbf{L}_\mathrm{LE}$ is sound and complete with respect to its corresponding class of algebras $\mathbb{LE}$, i.e.~that any sequent $\varphi\vdash\psi$ is provable in $\mathbf{L}_\mathrm{LE}$ if and only if $\mathbb{LE}\models\varphi\vdash\psi$. %Moreover, it is not hard to see that every consistent LE-logic is characterized by the class of algebras for it.
\fi

\paragraph{\textbf{The fully residuated language $\mathcal{L}_\mathrm{LE}^*$.}}
\label{ssec:expanded language}
Any  language $\mathcal{L}_\mathrm{LE} = \mathcal{L}_\mathrm{LE}(\mathcal{F}, \mathcal{G})$ can be associated with the language $\mathcal{L}_\mathrm{LE}^* = \mathcal{L}_\mathrm{LE}(\mathcal{F}^*, \mathcal{G}^*)$, where $\mathcal{F}^*\supseteq \mathcal{F}$ and $\mathcal{G}^*\supseteq \mathcal{G}$ are obtained by expanding $\mathcal{L}_\mathrm{LE}$ with: 
\begin{enumerate}
	\item an $n_f$-ary connective $f^\sharp_i$ for $1 \leq i \leq n_f$, interpreted as the {\em right residual} of $f\in\mathcal{F}$ in its $i$th coordinate if $\varepsilon_{f}(i) = 1$ (resp.~its {\em Galois-adjoint} if $\varepsilon_{f}(i) = \partial$);
    %\footnote{That is,  the interpretation of $f^\sharp_i$ in any (residuated) LE $\mathbb{A}$ is an $n_f$-ary operation on $\mathbb{A}$ such that, for any $a_1,\ldots, a_{n_f}, b\in \mathbb{A}$, \[f(a_1,\ldots, a_i,\ldots, a_{n_f})\leq b \quad\text{ iff }\quad a_i\leq f^\sharp_i(a_1,\ldots, b,\ldots, a_{n_f}) \quad \text{ if } \varepsilon_{f}(i) = 1\]
	%\[f(a_1,\ldots, a_i,\ldots, a_{n_f})\leq b \quad\text{ iff } \quad f^\sharp_i(a_1,\ldots, b,\ldots, a_{n_f})\leq a_i \quad \text{ if } \varepsilon_{f}(i) = \partial\]}
	\item an $n_g$-ary connective $g^\flat_i$ for $1 \leq i \leq n_g$, interpreted as  the {\em left residual} of $g\in\mathcal{G}$ in its $i$th coordinate if $\varepsilon_{g}(i) = 1$ (resp.~its {\em Galois-adjoint} if $\varepsilon_{g}(i) = \partial$).%\footnote{That is,  the interpretation of $g^\flat_i$ in any (residuated) LE $\mathbb{A}$ is an $n_g$-ary operation on $\mathbb{A}$ such that, for any $a_1,\ldots, a_{n_f}, b\in \mathbb{A}$, \[b\leq g(a_1,\ldots, a_i,\ldots, a_{n_f}) \quad\text{ iff } \quad g^\flat_i(a_1,\ldots, b,\ldots, a_{n_f})\leq a_i \quad \text{ if } \varepsilon_{g}(i) = 1\]
	%\[b\leq g(a_1,\ldots, a_i,\ldots, a_{n_f}) \quad \text{ iff } \quad a_i\leq g^\flat_i(a_1,\ldots, b,\ldots, a_{n_f}) \quad \text{ if } \varepsilon_{g}(i) = \partial\]}
	% $ g^\flat_j$ for each and $g\in \mathcal{G}$, where and $0\leq j\leq n_g$ ($f^\sharp_i$ is the right residual of $f$ in the $i$-th coordinate, and $g^\flat_j$ is the left residual of $g$ in the $j$-th coordinate).
	%\footnote{The adjoints of the unary connectives $\Box$, $\Diamond$, $\downarrow$ and $\uparrow$ are denoted $\Diamondblack$, $\blacksquare$, $\blhd$ and $\brhd$, respectively.}
\end{enumerate}
We let
$f^\sharp_i\in\mathcal{G}^*$ if $\varepsilon_{f}(i) = 1$, and $f^\sharp_i\in\mathcal{F}^*$ if $\varepsilon_{f}(i) = \partial$. Dually, $g^\flat_i\in\mathcal{F}^*$ if $\varepsilon_{g}(i) = 1$, and $g^\flat_i\in\mathcal{G}^*$ if $\varepsilon_{g}(i) = \partial$. The order-type of additional connectives agrees with that of their intended interpretations. That is, for any $f\in \mathcal{F}$ and $g\in\mathcal{G}$, 
%each $g^\flat_j\in\mathcal{F}$, for each coordinate $i$ in $f$ or $g$,
	\begin{itemize}
		\item $\epsilon_{f_i^\sharp}(i) = \epsilon_{f}(i)$ and $\epsilon_{f_i^\sharp}(j) = \epsilon_{f}(j) \cdot \epsilon_{f}^\partial(i)$ for any $j\neq i$,
		\item $\epsilon_{g_i^\flat}(i) = \epsilon_{g}(i)$ and $\epsilon_{g_i^\flat}(j) = \epsilon_{g}(j) \cdot \epsilon_{g}^\partial(i)$ for any $j\neq i$.
	\end{itemize}
    \remove{
    \begin{enumerate}
		\item if $\epsilon_f(i) = 1$, then $\epsilon_{f_i^\sharp}(i) = 1$ and $\epsilon_{f_i^\sharp}(j) = \epsilon_f^\partial(j)$ for any $j\neq i$.
		\item if $\epsilon_f(i) = \partial$, then $\epsilon_{f_i^\sharp}(i) = \partial$ and $\epsilon_{f_i^\sharp}(j) = \epsilon_f(j)$ for any $j\neq i$.
		\item if $\epsilon_g(i) = 1$, then $\epsilon_{g_i^\flat}(i) = 1$ and $\epsilon_{g_i^\flat}(j) = \epsilon_g^\partial(j)$ for any $j\neq i$.
		\item if $\epsilon_g(i) = \partial$, then $\epsilon_{g_i^\flat}(i) = \partial$ and $\epsilon_{g_i^\flat}(j) = \epsilon_g(j)$ for any $j\neq i$.
	\end{enumerate}
}
where the product between $1$ and $\partial$ is isomorphic to the restriction of the product on integers to $\{1, -1\}$. For instance, if $f$ and $g$ are binary connectives such that $\varepsilon_f = (1, \partial)$ and $\varepsilon_g = (\partial, 1)$, then $\varepsilon_{f^\sharp_1} = (1, 1)$, $\varepsilon_{f^\sharp_2} = (1, \partial)$, $\varepsilon_{g^\flat_1} = (\partial, 1)$ and $\varepsilon_{g^\flat_2} = (1, 1)$.\footnote{This notation depends on the choice of the primitive connective and must be carefully adapted to well-known cases. For instance, the `fusion' connective $\circ$ (which, when denoted  as $f$, is such that $\varepsilon_f = (1, 1)$) has residuals
	$f_1^\sharp$ and $f_2^\sharp$, usually written as $/$ and
	$\backslash$ respectively. However, if $\backslash$ is taken as the primitive $g$, then $g_2^\flat$ is $\circ = f$, and
	$g_1^\flat(x_1, x_2): = x_2/x_1 = f_1^\sharp (x_2, x_1)$. This example shows
	that, when identifying $g_1^\flat$ and $f_1^\sharp$,  the conventional order of the coordinates is not preserved and depends on the chosen primitive.}

\begin{definition}\label{def:tense lattice logic}
	For any language $\mathcal{L}_\mathrm{LE}(\mathcal{F}, \mathcal{G})$, its associated basic $\mathcal{L}_\mathrm{LE}^\ast$-{\em logic} is defined by specializing Definition \ref{def:LE:logic:general} to the language $\mathcal{L}_\mathrm{LE}^* = \mathcal{L}_\mathrm{LE}(\mathcal{F}^*, \mathcal{G}^*)$ %a set of sequents $\varphi\vdash\psi$ with $\varphi,\psi\in\mathcal{L}_\mathrm{LE}^*$, which contains the axioms of the LE-logic $\mathbb{L}_\mathrm{LE}$, and is closed under rules for LE-logics plus
	and closing under the following additional residuation rules for $f\in \mathcal{F}$ and $g\in \mathcal{G}$:
	$$
	\begin{array}{cc}
	\AxiomC{$f(\varphi_1,\ldots,\varphi,\ldots, \varphi_{n_f}) \vdash \psi$}
	\doubleLine
	%			\RightLabel{$( = 1)$}
	\UnaryInfC{$\varphi\vdash^{\epsilon_{f}(i)} f^\sharp_i(\varphi_1,\ldots,\psi,\ldots,\varphi_{n_f})$}
	\DisplayProof
	&\qquad\qquad
	\AxiomC{$\varphi \vdash g(\varphi_1,\ldots,\psi,\ldots,\varphi_{n_g})$}
	\doubleLine
	%			\RightLabel{$(\epsilon_g(i) = 1)$}
	\UnaryInfC{$g^\flat_i(\varphi_1,\ldots, \varphi,\ldots, \varphi_{n_g})\vdash^{\epsilon_{g}(i)} \psi$}
	\DisplayProof
	\end{array}
	$$
	%			$$
	%			\begin{array}{cc}
	%			\AxiomC{$f(\varphi_1,\ldots,\varphi,\ldots, \varphi_{n}) \vdash \psi$}
	%			\doubleLine
	%			\RightLabel{$(\epsilon_f(i) = \partial)$}
	%			\UnaryInfC{$f^\sharp_i(\varphi_1,\ldots,\psi,\ldots,\varphi_{n})\vdash \varphi$}
	%			\DisplayProof
	%			&\quad
	%			\AxiomC{$\varphi \vdash g(\varphi_1,\ldots,\psi,\ldots,\varphi_{n})$}
	%			\doubleLine
	%			\RightLabel{($\epsilon_g(i) = \partial)$}
	%			\UnaryInfC{$\psi\vdash g^\flat_i(\varphi_1,\ldots, \varphi,\ldots, \varphi_{n})$}
	%			\DisplayProof
	%			\end{array}
	%			$$	
	The double line indicates that the rules above can be applied top-down and bottom-up. %in each rule above is a shorthand for two separate rules: one where the upper sequent serves as the premise and the lower as the conclusion, and another where the lower sequent is the premise and the upper is the conclusion.
\end{definition}
\paragraph{\textbf{Display calculi for basic normal LE-logics.}} %\label{ssec:syntactic frames associated w algebras}
Let $\mathcal{L} = \mathcal{L} (\mathcal{F}, \mathcal{G})$ be a fixed but arbitrary LE-signature.
%
%In this section we let $\mathcal{L} = \mathcal{L}(\mathcal{F}, \mathcal{G})$ be a fixed but arbitrary LE-signature (cf.~Section \ref{subset:language:algsemantics}.1) and define the
 %display calculus $\mathrm{D.LE}$ for the basic normal $\mathcal{L}_{LE}$-logic and the display calculus $\mathrm{D.LE^\ast}$ for the fully residuated normal $\rbL^\ast_{LE}$ generated by the basic normal $\mathcal{L}_{LE}$-logic simultaneously. Their cut-free counterparts are denoted by $\mathrm{\cfDLE}$ and  $\mathrm{\cfDLE^\ast}$, respectively. %display calculus $\mathrm{D.LE^\ast}$ for the fully residuated normal $\rbL^\ast_{LE}$ generated by the basic normal $\mathcal{L}_{LE}$-logic and and its cut-free counterpart $\mathrm{\cfDLE^\ast}$. The display calculus $\mathrm{D.LE}$ for the basic normal $\mathcal{L}_{LE}$-logic is obtained from $\mathrm{D.LE^\ast}$ by deleting the logical (also called operational) rules for the operators that do not occur in the original LE-signature. 
 %The display calculus $\mathrm{D.LE^\ast}$ is obtained from $\mathrm{D.LE}$ by adding the logical (also called operational) rules for the operators that do not occur in the original LE-signature. 
 %As is usual of existing logical systems which the present framework intends to capture (e.g.~intuitionistic and bi-intuitionistic logics, or modal and tense logics \cite{gore1998substructural}), the languages manipulated by these calculi are built up using {\em one and the same} set of structural terms, and differ only in the set of operational (also called logical) term constructors. 
 Let $S_{\mathcal{F}} \ceq \{\FH \mid f\in \mathcal{F}^*\}$ and $S_{\mathcal{G}}\ceq \{\GC \mid g\in \mathcal{G}^*\}$ be the sets of {\em structural connectives} associated with  $\mathcal{F}^*$ and $ \mathcal{G}^*$ respectively.\footnote{For any connective $h$ of arity $n \geq 1$, the symbol $\hat{h}$ (resp.~$\check{h}$) conveys  the information that $h$ is a left (resp.~right) adjoint/residual.} Each  structural connective inherits the arity and  order-type  of its associated operational connective in $ \mathcal{F}^*$ and $\mathcal{G}^*$.

For  convenience, we set $\mathcal F^\partial:=\mathcal G$ and 
$\mathcal G^\partial:=\mathcal F$, and for any order-type $\epsilon$ on $n$, we let $\mathsf{Str}_\mathcal{F}^{\epsilon} : = \prod_{i = 1}^{n}\mathsf{Str}_{\mathcal{F}^{\epsilon(i)}}$ and $\mathsf{Str}_\mathcal{G}^{\epsilon} : = \prod_{i = 1}^{n}\mathsf{Str}_{\mathcal{G}^{\epsilon(i)}}$, where the sets $\mathsf{Str}_\mathcal F$ (resp.~$\mathsf{Str}_\mathcal G$) of {\em precedent} (resp.~{\em succedent}) structures   are defined recursively as follows:  
%
%, where for all $1 \leq i \leq n$,
%
%\begin{center}
%\begin{tabular}{ll}
%$\mathsf{Str}_\mathcal{F}^{\epsilon(i)} = \begin{cases} 
%\mathsf{Str}_\mathcal{F} &\mbox{ if } \epsilon(i) = 1\\
%\mathsf{Str}_\mathcal{G} &\mbox{ if } \epsilon(i) = \partial
%\end{cases}\quad$
%&
%$\mathsf{Str}_\mathcal{G}^{\epsilon(i)} = \begin{cases}
%\mathsf{Str}_\mathcal{G}& \mbox{ if } \epsilon(i) = 1,\\
%\mathsf{Str}_\mathcal{F} & \mbox{ if } \epsilon(i) = \partial.
%\end{cases}$
%\end{tabular}
%\end{center}
%The calculi $\mathrm{D.LE}$ manipulates $\mathcal{L}_\mathrm{LE}$-formulas:
%\begin{align*}
%\mathcal{L}_\mathrm{LE}\ni \varphi \ & ::=\ p \mid \bot \mid \top \mid  \varphi \vee \varphi \mid \varphi \wedge \varphi \mid f (\varphi_1, \ldots, \varphi_{n_f}) \mid g (\varphi_1, \ldots, \varphi_{n_g}) 
%\end{align*}
%where $p$ is an atomic formula and $f \in \mathcal{F}$ and $g \in \mathcal{G}$.

 %\marginnote{AP: this introduces a bit of clash of notation with the way we use $\Gamma$ and $\Delta$ in section 3. I would suggest to write $\Gamma\vdash \Delta$ as $\Pi\vdash \Sigma$, where $\Pi$ stands for `precedent', and $\Sigma$ stands for `succedent', but then we need to find another capital Greek letter also for the neutral one (how about $\Upsilon$? Or shall we go back to\\ $X$ for precedent structures and\\ $Y$ for succedent ones? G: I implemented your first suggestion.}

\begin{center}
	\begin{tabular}{@{}r@{}l@{}@{}r@{}l@{}}
		$\rule[-1.2ex]{0pt}{0ex}\mathsf{Str}_\mathcal{F} \ni \Pi$ \ & $ \cceq  \varphi \mid \AATOP \mid \FH\, (\overline{\Pi}^{(\varepsilon_f)})$ & $\hspace{0.5cm}\rule[-1.2ex]{0pt}{0ex}\mathsf{Str}_\mathcal{G} \ni \Sigma$ \ & $ \cceq  \varphi \mid \ABOT \mid \GC\, (\overline{\Sigma}^{(\varepsilon_g)})$  \\
	\end{tabular}
\end{center}
with $\varphi\in \mathcal{L}_\mathrm{LE}$, and $\FH \in S_{\mathcal{F}}$, $\GC\in S_{\mathcal{G}}$, $\overline{\Pi}^{(\varepsilon_f)}\in \mathsf{Str}_\mathcal{F}^{\epsilon_f}$ and $\overline{\Sigma}^{(\varepsilon_g)} \in \mathsf{Str}_\mathcal{G}^{\epsilon_g}$.  
%with $\Delta^{1} = \Delta$ and $\Delta^{\partial} = \Gamma$. 

% for all $1 \leq i \leq n_f$,
%\[
%x^{\epsilon_f(i)} \in \begin{cases} 
%\mathsf{Str}_\mathcal{F}&\mbox{ if } \epsilon_f(i) = 1\\
%\mathsf{Str}_\mathcal{G} &\mbox{ if } \epsilon_f(i) = \partial
%\end{cases}
%\] 
%and for all $1 \leq i \leq n_g$,
%\[
%y^{\epsilon_g(i)} \in \begin{cases}
%\mathsf{Str}_\mathcal{G}& \mbox{ if } \epsilon_g(i) = 1,\\
%\mathsf{Str}_\mathcal{F} & \mbox{ if } \epsilon_g(i) = \partial.
%\end{cases}
%\]
%In what follows, we let $\bari{x} := (x_1, \ldots, x_{i-1}, x_{i+1},\ldots, x_n)$ and $\bari{x}_z := (x_1, \ldots, x_{i-1}, z, x_{i+1},\ldots, x_n)$. $\bari{y}$ and $\bari{y}_z$ are defined likewise. 

In what follows, we use $\Upsilon_1, \ldots, \Upsilon_n$ as structure metavariables in $\mathsf{Str} \coloneqq \mathsf{Str}_\mathcal{F} \cup \mathsf{Str}_\mathcal{G}$. The introduction rules of the calculus below ensures that $\Upsilon \in \mathsf{Str}_\mathcal{F}$ (resp.~$\Upsilon \in \mathsf{Str}_\mathcal{G}$) whenever it occurs in precedent (resp.~succedent) position. 
The calculus $\mathrm{D.LE} = \mathrm{D.LE}_{\mathcal{L}}$  manipulates sequents $\Pi \vdash \Sigma$, and consists of the following rules: %\footnote{For any LE-language $\mathcal{L}$, we sometimes let $\mathrm{D.LE}^\ast \ceq  \mathrm{D.LE}_{\mathcal{L}^\ast}$, i.e. $\mathrm{D.LE}^\ast$ denote the calculus  obtained by instantiating the general definition of the basic calculus $\mathrm{D.LE}_{\mathcal{L}}$ to $\mathcal{L} \ceq \mathcal{L}^\ast$.}

\begin{itemize}
	\item Identity and cut rules:%\footnote{In the display calculi literature, the identity rule is sometimes defined as $\varphi \fCenter \varphi$, where $\varphi$ is an arbitrary, possibly complex, formula. The difference is inessential, given that, in any display calculus, $p \fCenter p$ is an instance of $\varphi \fCenter \varphi$, and $\varphi \fCenter \varphi$ is derivable for any formula $\varphi$ whenever $p \fCenter p$ is the Identity rule.}
\end{itemize}
\begin{center}
	\begin{tabular}{rl}
		\AXC{\phantom{$\Gamma \fCenter \varphi$}}
		\LL{\fns Id}
		\UI$p \fCenter p$
		\DP
		& 
		\AX$\Pi \fCenter \varphi$
		\AX$\varphi \fCenter \Sigma$
		\RL{\fns Cut}
		\BI$\Pi \fCenter \Sigma$
		\DP
		\\
	\end{tabular}
\end{center}

%\item Display rules: 
%$$
%			\begin{array}{cc}
%			\AX$\FH (\obx) \fCenter y$
%			\doubleLine
%			\LL{$(\epsilon_{f,i} = 1)$}
%			\UI$x_i \fCenter \G_{f^\sharp_{i}}(\bari{x}_y)$
%			\DP
%			&
%			\quad
%			\AxiomC{$\Gamma \Rightarrow \GC(\ory)$}
%			\doubleLine
%			\RightLabel{$(\epsilon_{g,i} = 1)$}
%			\UnaryInfC{$\F_{g^\flat_{i}}(\bari{y}_x)\Rightarrow y_i$}
%			\DisplayProof
%			\end{array}
%			$$
%			$$
%			\begin{array}{cc}
%			\AxiomC{$\FH(\obx) \Rightarrow y$}
%			\doubleLine
%			\LeftLabel{$(\epsilon_{f,i} = \partial)$}
%			\UnaryInfC{$\F_{f^\sharp_{i}}(\bari{x}_y)\Rightarrow x_i$}
%			\DisplayProof
%			&
%			\quad
%			\AxiomC{$x \Rightarrow \GC(\ory)$}
%			\doubleLine
%			\RightLabel{($\epsilon_{g,i} = \partial)$}
%			\UnaryInfC{$y_i \Rightarrow \G_{g^\flat_{i}}(\bari{y}_x)$}
%			\DisplayProof
%			\end{array}
%			$$

\begin{itemize}		
	\item Display postulates for $f\in \mathcal{F}$ and $g\in \mathcal{G}$: for any $1\leq i,j\leq n_f$ and $1\leq h,k\leq n_g$,
\end{itemize}
\begin{itemize}
	\item[] If $\varepsilon_{f}(i) = 1$ and $\varepsilon_{g}(h) = 1$,
\end{itemize}
\begin{center}
	\begin{tabular}{@{}c@{}c@{}}
		\AXC{$\Pi \fCenter \GC\, (\Upsilon_1 \ldots, \Sigma_h, \ldots \Upsilon_{n_g})$}
		\doubleLine
		\LL{\fns $\GHF_h \dashv \GC$}
		\UIC{$\GHF_h\, (\Upsilon_1, \ldots, \Pi, \ldots, \Upsilon_{n_g}) \fCenter \Sigma_h$}
		\DP 
         &
        \AXC{$\FH\, (\Upsilon_1, \ldots, \Pi_i, \ldots, \Upsilon_{n_f}) \fCenter \Sigma$}
		\doubleLine
		\RL{\fns $\FH \dashv \FCS_i$}
		\UIC{$\Pi_i \fCenter \FCS_i\, (\Upsilon_1, \ldots, \Sigma, \ldots, \Upsilon_{n_f})$}
		\DP \\
	\end{tabular}
\end{center}
\begin{itemize}
	\item[] If $\varepsilon_{f}(j) = \partial$ and $\varepsilon_{g}(k) = \partial$,
\end{itemize}
\begin{center}
	\begin{tabular}{@{}c@{}c@{}}					
		\AX$\FH\, (\Upsilon_1, \ldots, \Sigma_j, \ldots, \Upsilon_{n_f}) \fCenter \Sigma$
		\doubleLine
		\LL{\fns $(\FH, \FHS_j)$}
		\UI$\FHS_j\, (\Upsilon_1, \ldots, \Sigma, \ldots, \Upsilon_{n_f}) \fCenter \Sigma_j$
		\DP
		& 
		\AX$\Pi \fCenter \GC\, (\Upsilon_1, \ldots, \Pi_k, \ldots, \Upsilon_{n_g})$
		\doubleLine
		\RL{\fns $(\GC, \GCF_k)$}
		\UI$ \Pi_k \fCenter \GCF_k\, (\Upsilon_1, \ldots, \Pi, \ldots, \Upsilon_{n_g})$
		\DP \\
	\end{tabular}
\end{center}

The structural connectives of the form $\FHS$ and $\GCF$ are referred to as {\em residuals}.

\begin{itemize}
	\item Structural rules for lattice connectives:
\end{itemize}
\begin{center}
	\begin{tabular}{rl}
		\AX$\AATOP \fCenter \Sigma$
		\LL{\fns $\aatop_W$}
		\UI$\Pi \fCenter \Sigma$
		\DP
		& 
		\AX$\Pi \fCenter \ABOT$
		\RL{\fns $\abot_W$}
		\UI$\Pi \fCenter \Sigma$
		\DP
	\end{tabular}
\end{center}

\begin{itemize}
	\item Logical introduction rules for lattice connectives:
\end{itemize}
\begin{center}
	\begin{tabular}{rl}
		\AX$\AATOP \fCenter \Sigma$
		\LL{\fns $\aatop_L$}
		\UI$\aatop \fCenter \Sigma$
		\DP
		\ 
		\AXC{$\phantom{\AATOP \fCenter \Sigma}$}
		\RL{\fns $\aatop_R$}
		\UI$\AATOP \fCenter \aatop$
		\DP
		& 
		\AXC{$\phantom{\Pi \fCenter \ABOT}$}
		\LL{\fns $\abot_L$}
		\UI$\abot \fCenter \ABOT$
		\DP
		\ 
		\AX$\Pi \fCenter \ABOT$
		\RL{\fns $\abot_R$}
		\UI$\Pi \fCenter \abot$
		\DP
		\\
		
		& \\
		
		\AX$\psi \fCenter \Sigma$
		\LL{\fns $\aand_{L2}$}
		\UI$\varphi \aand \psi \fCenter \Sigma$
		\DP
		\ 
		\AX$\varphi \fCenter \Sigma$
		\LL{\fns $\aand_{L1}$}
		\UI$\varphi \aand \psi \fCenter \Sigma$
		\DP
		& 
		\AX$\Pi \fCenter \varphi$
		\AX$\Pi \fCenter \psi$
		\RL{\fns $\aand_R$}
		\BI$\Pi \fCenter \varphi \aand \psi$
		\DP
		\\
		
		& \\
		
		\AX$\varphi \fCenter \Sigma$
		\AX$\psi \fCenter \Sigma$
		\LL{\fns $\aor_L$}
		\BI$\varphi \aor \psi \fCenter \Sigma$
		\DP
		& 
		\AX$\Pi \fCenter \varphi$
		\RL{\fns $\aor_{R1}$}
		\UI$\Pi \fCenter \varphi \aor \psi$
		\DP
		\ 
		\AX$\Pi \fCenter \psi$
		\RL{\fns $\aor_{R2}$}
		\UI$\Pi \fCenter \varphi \aor \psi$
		\DP
		\\
	\end{tabular}
	
\end{center}

\begin{itemize}
	\item Logical introduction rules for $f\in\mathcal{F}$ and $g\in\mathcal{G}$:
\end{itemize}

\begin{center}
    \begin{tabular}{c c}
         \bottomAlignProof
                        \AxiomC{$\Big(\varphi_i \fCenter^{\!\!\epsilon_{{g}}(i)}\; \Upsilon_i \,\mid\, 1\leq i\leq n_g \Big)$}
						\LL{\fns$g_L$}
						\UI$g(\overline{\varphi}) \fCenter \GC\, (\overline{\Upsilon})$
						\DP
						
         &  
         \bottomAlignProof
						
                        \AxiomC{$\Big(\Upsilon_i \fCenter^{\!\!\epsilon_{f}(i)}\, \varphi_i  \mid 1\leq i\leq n_f\Big)$}
						\RL{\fns$f_R$}
						\UI$\FH\, (\overline{\Upsilon})\fCenter f(\overline{\varphi})$
						\DP
    \end{tabular}
\end{center}

\remove{
\begin{center}
	\begin{tabular}{c}
		\bottomAlignProof
		\AxiomC{$\Big(\Upsilon_i \fCenter \varphi_i \quad \varphi_j \fCenter \Upsilon_j \mid 1\leq i, j\leq n_f, \varepsilon_{f}(i) = 1\mbox{ and } \varepsilon_{f}(j) = \partial\Big)$}
		\RL{\fns$f_R$}
		\UI$\FH\, (\Upsilon_1,\ldots, \Upsilon_{n_f})\fCenter f(\varphi_1,\ldots, \varphi_{n_f})$
		\DP
	\end{tabular}
\end{center}
\begin{center}
	\begin{tabular}{c}
		\bottomAlignProof
		\AxiomC{$\Big( \varphi_i \fCenter \Upsilon_i \quad \Upsilon_j \fCenter \varphi_j \,\mid\, 1\leq i, j\leq n_g, \varepsilon_{g}(i) = 1\mbox{ and } \varepsilon_{g}(j) = \partial \Big)$}
		\RL{\fns$g_L$}
		\UI$g(\varphi_1,\ldots, \varphi_{n_g}) \fCenter \GC\, (\Upsilon_1,\ldots, \Upsilon_{n_g})$
		\DP
	\end{tabular}
\end{center}
}
\begin{center}
	\begin{tabular}{c c}
		\bottomAlignProof
		\AX$\FH\, (\varphi_1,\ldots, \varphi_{n_f}) \fCenter \Sigma$
		\LL{\fns$f_L$}
		\UI$f(\varphi_1,\ldots, \varphi_{n_f}) \fCenter \Sigma$
		\DP
		&
		\bottomAlignProof
		\AX$\Pi \fCenter \GC\, (\varphi_1,\ldots, \varphi_{n_g})$
		\RL{\fns$g_R$}
		\UI$\Pi \fCenter g(\varphi_1,\ldots, \varphi_{n_g})$
		\DP
	\end{tabular}
\end{center}

%where $i$-monotone ($j$-antitone) means that the operator is order preserving (order reversing) at the $i$-th ($j$-th) coordinate.
If $f$ and $g$ are $0$-ary (i.e.~they are constants), the rules $f_R$ and $g_L$ above reduce to the axioms (aka $0$-ary rules) $\FH \fCenter f$ and $g \fCenter \GC$.

	The calculus $\mathrm{D.LE}$ %(hence also $\mathrm{\cfDLE}$) 
    is sound and complete w.r.t.~the class of  complete  $\mathcal{L}$-algebras \cite{greco2018algebraic}.\footnote{\label{footnote:complete algebras} A {\em complete}  $\mathcal{L}$-algebra is a complete lattice endowed with $n_f$-ary (resp.~$n_g$-ary) operations $f$ (resp.~$g$) for each $f\in \mathcal{F}$ (resp.~$g\in \mathcal{G}$) which are completely join-preserving (resp.~meet-preserving) in any coordinate $i$ s.t.~$\varepsilon_f(i) = 1$ (resp.~$\varepsilon_g(i) = 1$), and completely meet-reversing (resp.~join-reversing) in any coordinate $i$ s.t.~$\varepsilon_f(i) = \partial$ (resp.~$\varepsilon_g(i) = \partial$). By well known order theoretic facts, complete algebras are also  {\em completely residuated}, that is, the residuals $f^\sharp_i$ and $g^\flat_h$ exist for any $f\in \mathcal{F}$ and $g\in\mathcal{G}$ in each coordinate.
}
%
%
%\begin{proof}
\remove{
	The soundness of the basic lattice rules is clear. The soundness of the remaining rules is due to the monotonicity (resp.~antitonicity) of the algebraic connectives interpreting each $f\in \mathcal{F}$ and $g\in \mathcal{G}$, and their adjunction/residuation properties, which hold since any complete $\mathcal{L}$-algebra is an $\mathcal{L}^*$-algebra.}
%\end{proof}
\frenchspacing
Moreover, it	is a proper display calculus (cf.~\cite[Theorem 26]{GMPTZ}), and hence cut elimination holds for it as a consequence of a Belnap-style cut elimination metatheorem (cf.~\cite[Section 2.2 and Appendix A]{GMPTZ} and \cite[Theorem 2]{GreLianMosPal2020}).	%\end{prop}

\section{Refutation display calculi for LE-logics}\label{sec:refutational}

\remove{\begin{definition}\label{def:modal_trace}
    Given a propositional atomic variable $p$ in a sequent $A\vdash B$, we define its \emph{modal trace} $\mathrm{mtr}(p)$ to be the ordered list of the modal connectives containing $p$, from the leaf $p$ to the root of the sequent, together with their polarity and the coordinate of the branch leading to $p$.
    %We write $-\mathrm{mtr}(p)$ to denote the list $\mathrm{mtr}(p)$ where all the entries have reversed polarity. 
    The modal trace of a logical connective is defined similarly.
\end{definition}

\begin{example}
    Consider the sequent $\Diamond(\Box q \rightarrow r)\vdash\Box\Diamond p \rightarrow \Box r$. We have:
    \begin{itemize}
        \item $\mathrm{mtr}(p) = [(\Diamond, \partial,1),(\Box,\partial,1),(\rightarrow,1,1)]$,
        %\item $-\mathrm{mtr}(p) = [(\Diamond, 1,1),(\Box,1,1),(\rightarrow,\partial,1)]$,
        \item $\mathrm{mtr}(q)=[(\Box,1,1),(\rightarrow,\partial,1),(\Diamond,\partial,1)]$,
        \item $\mathrm{mtr}()$
    \end{itemize}
\end{example}}

%\marginnote{AP: bisogna aggiungere la def di positive/negative polarity. Nota che qui positive e negative sono al contrario di come le definiamo in alba. spiegare in un remark che questa e' simile alla definite perche' vogliamo che non ci siano delta adjoints nello skeleton delle formule appese alla struttura col segno che ereditano e portare su nealla sezione 3?\textcolor{red}{ADD: decidete voi il modo per gestire questa definizione}}
\begin{definition} \label{def:branching}
A sequent $\Pi\vdash\Sigma$ is \emph{branching}\footnote{The definition of branching sequents captures    {\em non-definite} inequalities in \cite[Definition 3.2]{CoPa11}; the  difference is that, here, both structural and operational connectives are considered. Hence, every branching sequent is equivalent to a set of non-branching ones (cf.~\cite[Lemma 8.2]{CoPa11}).} if the  signed generation trees\footnote{In the context of a sequent $\Pi\vdash \Sigma$,  the  {\em signed generation trees} of $\Pi$ and $\Sigma$ are defined by assigning $+$ (resp.~$-$) to the root of the generation tree of $\Pi$ (resp.~$\Sigma$), and then propagating the sign according to the order-type of both structural and operational connectives in the generation trees of $\Pi$ and $\Sigma$ and of the formulas therein: a child-node corresponding to a coordinate of the parent node of order-type $1$ (resp.~$\partial$) inherits the same sign as (resp.~opposite sign to) the sign of the parent node. 
    } of  $\Pi$ and $\Sigma$  contain some  $+\vee$ or  $-\wedge$ node such that all nodes (if any) connecting it to the root are labelled in the set $\{+ f, +\FH, -g, -\GC\mid f\in\mathcal{F}, g\in\mathcal{G}\}$. %branch connecting it to  which occurs in the scope of (possibly, $0$) $\mathcal{F}$ operators of positive (resp.~negative) polarity and (possibly, $0$) $\mathcal{G}$ operators of negative (resp.~positive) polarity.
\end{definition}

The refutation calculus D.LE$^r$ manipulates {\em antisequents} -- i.e., syntactic objects of the form $\Pi\nvdash \Sigma$, where $\Pi\in \mathsf{Str}_\mathcal{F}^{\epsilon}$ and $\Sigma\in \mathsf{Str}_\mathcal{G}^{\epsilon}$. Intuitively, an antisequent $\Pi\nvdash \Sigma$ is derivable if and only if the sequent $\Pi\vdash\Sigma$ is invalid.\footnote{In writing antisequents, we employ $\nvdash$ rather than the standard $\dashv$ to avoid a clash of notation with the symbol denoting adjunction.} In what follows, we list the rules of D.LE$^r$ using the same notational conventions introduced in the previous section, with the addition that $\varphi\nvdash^\partial \psi$ stands for $\psi \nvdash \varphi$, and $\check{\bot}^\partial = \hat{\top}$ and $\hat{\top}^\partial = \check{\bot}$.
\begin{itemize}
    \item Axiomatic rules:
\end{itemize}
    \begin{center}
    \begin{tabular}{c cc c}
        {\AxiomC{}
         \LL{\scriptsize{Ax1}}
        \UnaryInfC{$\AATOP \nvdash \ABOT$}
                \DisplayProof}\qquad
                {\AxiomC{}
         \LL{\scriptsize{Ax2}}
        \UnaryInfC{$p \nvdash \ABOT$}
        \DisplayProof}\qquad
        {\AxiomC{}
         \RL{\scriptsize{Ax3}}
        \UnaryInfC{$\AATOP \nvdash q$}
        \DisplayProof}\qquad
        {\AxiomC{}
         \RL{\scriptsize{Ax4}}
        \UnaryInfC{$p \nvdash q$}
        \DisplayProof}
	\end{tabular}
\end{center}
\begin{itemize}
\item Display rules for $f\in \mathcal{F}$ and $g\in \mathcal{G}$: for any $1\leq i \leq n_f$ and $1\leq j \leq n_g$,
\end{itemize}

\begin{itemize}
	\item[] if $\varepsilon_{f}(i) = 1$ and $\varepsilon_{g}(j) = 1$,%\footnote{The notation $\FH \nvdash \FCS_i$ (resp.~$\GHF_j \nvdash \GC$) indicates that $\FH$ and $\FCS_i$ (resp.~$\GHF_j$ and $\GC$) are in a {\em residuated pair} and $\FCS_i$ (resp.~$\GHF_j$) is the right residual (resp.~left residual) of $\FH$ (resp.~$\GC$) in the $i$-th coordinate (resp.~$j$-th coordinate).}
\end{itemize}
\begin{center}
	\begin{tabular}{lr}
		\AXC{$\Pi \nvdash \GC\, (\Upsilon_1 \ldots, \Sigma_j, \ldots \Upsilon_{n_g})$}
		\doubleLine
		\LL{\fns $\GHF_j \dashv \GC$}
		\UIC{$\GHF_j\, (\Upsilon_1, \ldots, \Pi, \ldots, \Upsilon_{n_g}) \nvdash \Sigma_i$}
		\DP 
        &
        \AXC{$\FH\, (\Upsilon_1, \ldots, \Pi_i, \ldots, \Upsilon_{n_f}) \nvdash \Sigma$}
		\doubleLine
		\RL{\fns $\FH \dashv \FCS_i$}
		\UIC{$\Pi_i \nvdash \FCS_i\, (\Upsilon_i, \ldots, \Sigma, \ldots, \Upsilon_{n_f})$}
		\DP \\
	\end{tabular}
\end{center}
\begin{itemize}
	\item[] if $\varepsilon_{f}(i) = \partial$ and $\varepsilon_{g}(j) = \partial$, %\footnote{The notation $(\GC, \GCF_j)$ (resp.~$(\FH, \FHS_i)$) indicates that $\GC$ and $\GCF_j$ (resp.~$\FH$ and $\FHS_i$) are in a {\em Galois connection} (resp.~{\em dual Galois connection}) and $\GCF_j$ (resp.~$\FHS_i$) is the right residual (resp.~left residual) of $\GC$ (resp.~$\FH$) in the $j$-th coordinate (resp.~$i$-th coordinate).}
\end{itemize}
\begin{center}
	\begin{tabular}{lr}		
        \AXC{$\FH\, (\Upsilon_1, \ldots, \Sigma_i, \ldots, \Upsilon_{n_f}) \nvdash \Sigma$}
        \LL{}
        \doubleLine
        \LL{\fns $(\FH, \FHS_i)$}
        \UIC{$\FHS_i\, (\Upsilon_1, \ldots, \Sigma, \ldots, \Upsilon_{n_f}) \nvdash \Sigma_i$}
        \DP &
		\AXC{$\Pi \nvdash \GC\, (\Upsilon_1, \ldots, \Pi_j, \ldots, \Upsilon_{n_g})$}
        \doubleLine
		\RL{\fns $(\GC, \GCF_j)$}
        \UIC{$ \Pi_j \nvdash \GCF_j\, (\Upsilon_1, \ldots, \Pi, \ldots, \Upsilon_{n_g})$}
        \DP
	\end{tabular}
\end{center}
\begin{itemize}
\item Structural rules$^\ast$ for $f\in\mathcal{F}$ and $g\in\mathcal{G}$: 
\end{itemize}
\begin{center}
	\begin{tabular}{c c}
    \AXC{$(\Upsilon_i \nvdash^{\epsilon_f(i)} \ABOT^{\epsilon_f(i)} \mid 1\leq i\leq n_f) $}
    \LL{\fns$\FH\ABOT$}
    \UIC{$\FH\, (\Upsilon_1,\ldots, \Upsilon_{n_f}) \nvdash\ABOT$}
    \DP &
    \AXC{$( \AATOP^{\epsilon_g(i)}\nvdash^{\epsilon_g(i)} \Upsilon_i \mid 1\leq i\leq n_g ) $}
    \RL{\fns$\ATOP \GC$}
    \UIC{$\AATOP\nvdash\GC\, (\Upsilon_1,\ldots, \Upsilon_{n_g})$}
    \DP
	\end{tabular}
\end{center}

%\textcolor{red}{GG: It is worth observing that the rule above $\ATOP \GC$ (resp.~$\FH\ABOT$) makes necessitation (resp.~co-necessitation) derivable: just take $n_g=1$ (resp.~$n_f=1$).}

\begin{center}
	\begin{tabular}{c c}
    \AXC{$(\Upsilon_i \nvdash^{\epsilon_f(i)} \ABOT^{\epsilon_f(i)} \mid 1\leq i\leq n_f) $}
    \LL{\fns$\FH p$}
    \UIC{$\FH\, (\Upsilon_1,\ldots, \Upsilon_{n_f}) \nvdash p$}
    \DP &
    \AXC{$( \AATOP^{\epsilon_g(i)}\nvdash^{\epsilon_g(i)} \Upsilon_i \mid 1\leq i\leq n_g ) $}
    \RL{\fns$p\GC$}
    \UIC{$p \nvdash\GC\, (\Upsilon_1,\ldots, \Upsilon_{n_g})$}
    \DP
	\end{tabular}
\end{center}

\begin{center}
    \AXC{$(  \Upsilon_i \nvdash^{\epsilon_f(i)} \ABOT^{\epsilon_f(i)} \mid 1\leq i\leq n_f ) $}
    \AXC{$( \AATOP^{\epsilon_g(i)}\nvdash^{\epsilon_g(i)} \Upsilon'_i \mid 1\leq i\leq n_g ) $}
    \RL{\fns$\FH\GC$}
    \BIC{$\FH(\Upsilon_{1},\ldots,\Upsilon_{n_f})\nvdash\GC\, (\Upsilon'_1,\ldots, \Upsilon'_{n_g})$}
    \DP
\end{center}

$^\ast$Side condition: in each rule above, the endsequent should not contain residuals.
\begin{itemize}
\item Logical introduction rules$^{\ast\ast}$ for $f\in\mathcal{F}$ and $g\in\mathcal{G}$:
\end{itemize}
\begin{center}
	\begin{tabular}{c c c c}
    \AXC{}
    \LL{\fns$g\ABOT$}
    \UIC{$g(\overline{\varphi}) \nvdash \ABOT$}
    \DP &
    \AXC{}
    \LL{\fns$g p$}
    \UIC{$g(\overline{\varphi}) \nvdash p$}
    \DP &
    \AXC{}
    \RL{\fns$p f$}
    \UIC{$p\nvdash f(\overline{\varphi})$}
    \DP &
    \AXC{}
    \RL{\fns$\AATOP f$}
    \UIC{$\AATOP\nvdash f(\overline{\varphi})$}
    \DP
	\end{tabular}
\end{center}

\begin{center}
	\begin{tabular}{c c c}
    \AXC{$\FH\, (\overline{\varphi}) \nvdash\Sigma$}
    \LL{\fns$f_L$}
    \UIC{$f(\overline{\varphi}) \nvdash \Sigma$}
    \DP &
    \AXC{\phantom{$\FH\, (\overline{\varphi}) \nvdash\Sigma$}}
    \RL{\fns$gf$}
    \UIC{$g(\overline{\psi}) \nvdash f(\overline{\varphi})$}
    \DP &
    \AXC{$\Pi \nvdash \GC\, (\overline{\varphi})$}
    \RL{\fns$g_R$}
    \UIC{$\Pi \nvdash g(\overline{\varphi})$}
    \DP
	\end{tabular}
\end{center}

\begin{center}
    \AXC{$(\Upsilon_i \nvdash^{\epsilon_f(i)} \ABOT^{\epsilon_f(i)} \mid 1\leq i\leq n_f )$}
    \AXC{$\Upsilon_j \nvdash^{\epsilon_f(j)} \varphi_j$}
    \RL{\fns$f_R$}
    \BIC{$\FH\, (\Upsilon_1,\ldots, \Upsilon_{n_f}) \nvdash f(\varphi_1,\ldots,\varphi_{n_f})$}
    \DP 
\end{center}

\remove{
\textcolor{red}{GG: In the rule $f_R$ above the arguments of $\FH$, namely $\Upsilon_1,\ldots, \Upsilon_{n_f}$ are copied twice in the premises, vice versa the arguments of $f$, namely $\varphi_1,\ldots,\varphi_{n_f}$, are copied only once in the premises... so, this rule is not ``symmetric''. What about the following rule, where also $\varphi_1,\ldots,\varphi_{n_f}$ are copied twice? Is it sound?
\begin{center}
    \AXC{$(\Upsilon_i \nvdash^{\epsilon_f(i)} \ABOT^{\epsilon_f(i)} \mid 1\leq i\leq n_f )$}
    \AXC{$\Upsilon_j \nvdash^{\epsilon_f(j)} \varphi_j$}
    \AXC{$(\ATOP^{\epsilon_f(i)} \nvdash^{\epsilon_f(i)} \varphi_i \mid 1\leq i\leq n_f )$}
    \RL{\fns$f_R$}
    \TIC{$\FH\, (\Upsilon_1,\ldots, \Upsilon_{n_f}) \nvdash f(\varphi_1,\ldots,\varphi_{n_f})$}
    \DP 
\end{center}
A similar observation applies to the rule below $g_L$.
}
}

\begin{center}
    \AXC{$( \AATOP^{\epsilon_g(i)}\nvdash^{\epsilon_g(i)} \Upsilon_i \mid 1\leq i\leq n_g ) $}
    \AXC{$\varphi_j \nvdash^{\epsilon_g(j)} \Upsilon_j$}
    \LL{\fns$g_L$}
    \BIC{$g(\varphi_{1},\ldots,\varphi_{n_g})\nvdash\GC\, (\Upsilon_1,\ldots, \Upsilon_{n_g})$}
    \DP
\end{center}

\begin{center}
\begin{tabular}{cc}
    \AXC{$( \AATOP^{\epsilon_{g_2}(i)}\nvdash^{\epsilon_{g_2}(i)} \Upsilon_i \mid 1\leq i\leq n_{g_2} ) $}
    \LL{\fns$g_L^\neq$}    \UIC{$g_1(\varphi_{1},\ldots,\varphi_{n_{g_1}})\nvdash\GC_2\, (\Upsilon_1,\ldots, \Upsilon_{n_{g_2}})$}
    \DP
    &
    \AXC{$(\Upsilon_i \nvdash^{\epsilon_{f_1}(i)} \ABOT^{\epsilon_{f_1}(i)} \mid 1\leq i\leq n_{f_1} )$}
    \RL{\fns$f_R^\neq$}
    \UIC{$\FHp\, (\Upsilon_1,\ldots, \Upsilon_{n_{f_1}}) \nvdash f_2(\varphi_1,\ldots,\varphi_{n_{f_2}})$}
    \DP
\end{tabular}
\end{center}

\remove{
\textcolor{red}{GG: In the rule $f_R^\neq$ above the arguments of $\FH_1$, namely $\Upsilon_1,\ldots, \Upsilon_{n_{f_1}}$ are copied in the premises, vice versa the arguments of $f_2$, namely $\varphi_1,\ldots,\varphi_{n_{f_2}}$, are not copied at all in the premises... so, this rule is not ``linear''. What about the following rule, where also $\varphi_1,\ldots,\varphi_{n_{f_2}}$ are copied? Is it sound?
\begin{center}
\AXC{$(\Upsilon_i \nvdash^{\epsilon_{f_1}(i)} \ABOT^{\epsilon_{f_1}(i)} \mid 1\leq i\leq n_{f_1} )$}
\AXC{$(\ATOP^{\epsilon_{f_2}(i)} \nvdash^{\epsilon_{f_2}(i)} \Upsilon_i \mid 1\leq i\leq n_{f_2} )$}
\RL{\fns$f_R^\neq$}
\BIC{$\FHp\, (\Upsilon_1,\ldots, \Upsilon_{n_{f_1}}) \nvdash f_2(\varphi_1,\ldots,\varphi_{n_{f_2}})$}
\DP 
\end{center}}
}

Where $j$ in $f_R$ (resp.~$g_L$) is some index in $\{1,\ldots,n_f\}$ (resp.~$\{1,\ldots,n_g\}$).

$^{\ast\ast}$Side condition: $f_{1}\not=f_{2}$ and $g_{1}\not=g_{2}$. Furthermore, we require that the endsequents of the rules $f_R$, $f^\neq_R$, $g_L$, and $g^\neq_L$  not contain residuals.

\begin{itemize}
\item Logical introduction rules$^{\ast\ast\ast}$ for lattice connectives:
\end{itemize}
\begin{center}
	\begin{tabular}{c c c c}
    \AXC{$\AATOP\nvdash\Sigma$}
    \LL{\fns$\aatop_{L}$}
    \UIC{$\aatop\nvdash\Sigma$}
    \DP 
    &
    \AXC{$\varphi_i\nvdash \Sigma$}
    \LL{\fns$\vee_{L_i}$}
    \UIC{$\varphi_1\vee\varphi_2 \nvdash \Sigma$}
    \DP
    &
	\AXC{$\Pi\nvdash \varphi_i$}
    \RL{\fns$\wedge_{R_i}$}
    \UIC{$\Pi \nvdash \varphi_1\wedge\varphi_2$}
    \DP
    &
    \AXC{$\Pi\nvdash\ABOT$}
    \RL{\fns$\abot_{R}$}
    \UIC{$\Pi\nvdash\abot$}
    \DP
	\end{tabular}
\end{center}

\begin{center}
    \scriptsize{
    \AXC{$\varphi \nvdash \Sigma'\quad\psi \nvdash \Sigma'\quad (\varphi \wedge \psi \nvdash \Sigma'[A_i]^{pre})_{i\in I}\quad(\varphi \wedge \psi \nvdash \Sigma'[B_i]^{pre})_{i\in I}\quad(\varphi \wedge \psi \nvdash \Sigma'[C_j]^{suc})_{j\in J}\quad (\varphi \wedge \psi \nvdash \Sigma'[D_j]^{suc})_{j\in J}$}
    \LL{$\wedge_L$}
    \UIC{$\varphi \wedge \psi \nvdash \Sigma'[A_i\wedge B_i]_{i\in I}^{pre}[C_j\vee D_j]_{j\in J}^{suc}$}
    \DP}
\end{center}
\begin{center}
    \scriptsize{
    \AXC{$\Pi' \nvdash \varphi \quad \Pi' \nvdash \psi \quad (\Pi'[A_i]^{pre}\nvdash\varphi \vee \psi)_{i\in I}\quad(\Pi'[B_i]^{pre}\nvdash \varphi \vee \psi)_{i\in I}\quad(\Pi'[C_j]^{suc} \nvdash \varphi \vee \psi)_{j\in J}\quad (\Pi'[D_j]^{suc}\nvdash \varphi \vee \psi)_{j\in J}$}
    \RL{$\vee_R$}
    \UIC{$\Pi'[A_i\wedge B_i]_{i\in I}^{pre}[C_j\vee D_j]_{j\in J}^{suc} \nvdash \varphi \vee \psi$}
    \DP}
\end{center}

$^{\ast\ast\ast}$Side condition: $\Pi'$ and $\Sigma'$ are not branching (see Definition \ref{def:branching}). Furthermore, we require that the endsequents of the rules $\wedge_L$ and $\vee_R$  not contain residuals.

\begin{example}
    A $\mathrm{D.LE}^{r}$-derivation of the antisequent $g(p\vee q)\nvdash g(p)\vee g(q)$, for some $g\in\mathcal{G}$ with order-type $\langle1\rangle$:
    \begin{center}
        \AXC{$ $}
        \RL{\scriptsize{$Ax_{3}$}}
        \UIC{$\ATOP\nvdash p$}
        \AXC{$ $}
        \RL{\scriptsize{$Ax_{4}$}}
        \UIC{$q\nvdash p$}
        \LL{\scriptsize{$\vee_{L_{2}}$}}
        \UIC{$p\vee q\nvdash p$}
        \LL{\scriptsize{$g_{L}$}}
        \BIC{$g(p\vee q)\nvdash \GC(p)$}
        \RL{\scriptsize $g_R$}
        \UIC{$g(p\vee q)\nvdash g(p)$}
        
        \AXC{$ $}
        \RL{\scriptsize{$Ax_{3}$}}
        \UIC{$\ATOP\nvdash q$}
        \AXC{$ $}
        \RL{\scriptsize{$Ax_{4}$}}
        \UIC{$p\nvdash q$}
        \LL{\scriptsize{$\vee_{L_{1}}$}}
        \UIC{$p\vee q\nvdash q$}
        \LL{\scriptsize{$g_{L}$}}
        \BIC{$g(p\vee q)\nvdash \GC(q)$}
        \RL{\scriptsize$g_R$}
        \UIC{$g(p\vee q)\nvdash g(q)$}
        \RL{\scriptsize{$\vee_{R}$}}
        \BIC{$g(p\vee q)\nvdash g(p)\vee g(q)$}
        \DP
    \end{center}
\end{example}

\begin{example}
The side conditions are necessary for the soundness of the rules. Consider $g\in\mathcal{G}$ and $f \in \mathcal{F}$ both with order-type $\langle1\rangle$. Both $g(p) \wedge g(q) \vdash g(p \wedge q)$ and $f(p) \vdash f(p)$ are valid, hence the derivation below are not sound.
    {\scriptsize
    \begin{flushleft}
    \begin{tabular} {@{}c@{}}
    \!\!\!\!\!\!\!\!\!\!\!\!\!\!\!\!\!\!
        \AXC{$ $}
        \RL{\scriptsize{$Ax_{3}$}}
        \UIC{$\ATOP\nvdash q$}
        \RL{\scriptsize$\wedge_{R_2}$}
        \UIC{$\ATOP\nvdash p \wedge q$}
        
        \AXC{$ $}
        \RL{\scriptsize{$Ax_{4}$}}
        \UIC{$p\nvdash q$}
        \RL{\scriptsize{$\wedge_{R_{2}}$}}
        \UIC{$p\nvdash p \wedge q$}
        
        \LL{\scriptsize{$g_{L}$}}
        \BIC{$g(p) \nvdash \GC (p \wedge q)$}
        
        \AXC{$ $}
        \RL{\scriptsize{$Ax_{3}$}}
        \UIC{$\ATOP\nvdash p$}
        \RL{\scriptsize$\wedge_{R_1}$}
        \UIC{$\ATOP\nvdash p \wedge q$}
        \AXC{$ $}
        
        \RL{\scriptsize{$Ax_{4}$}}
        \UIC{$q\nvdash p$}
        \RL{\scriptsize{$\wedge_{R_{1}}$}}
        \UIC{$q\nvdash p\wedge q$}
        
        \LL{\scriptsize{$g_{L}$}}
        \BIC{$g(q) \nvdash \GC (p \wedge q)$}
        
        \LL{\scriptsize{$\wedge_{L}$}}
        \BIC{$g(p) \wedge g(q) \nvdash \GC (p \wedge q)$}
        \RL{\scriptsize$g_R$}
        \UIC{$g(p) \wedge g(q) \nvdash g(p \wedge q)$}
        \DP

        \AXC{}
        \RL{\scriptsize$\AATOP f$}
        \UIC{$\AATOP \nvdash f(p)$}
        \RL{\scriptsize$p \FC^\sharp_1$}
        \UIC{$p \nvdash \FC^\sharp_1(f(p))$}
        \LL{\scriptsize $(\FH, \FCS_i)$}
        \UIC{$\FH(p) \nvdash f(p)$}
        \LL{\scriptsize$f_L$}
        \UIC{$f(p) \nvdash f(p)$}
        \DP
        \end{tabular}
        \\
    \end{flushleft}
    }
     Notice that in the application of $\wedge_L$ the succedent is branching, while the endsequent of $p \FC^\sharp_1$ contains residuals.
\end{example}

\section{Soundness and completeness of the refutation calculus}\label{sec:soundness_completeness}

Let the {\em complexity} of a sequent be defined as the number of logical connectives plus the number of connectives (both logical and structural) plus twice the number of propositional atoms.

In what follows, we write $\Upsilon[\Upsilon']$ to indicate that $\Upsilon'$ is a substructure of $\Upsilon$. When $\Upsilon$ occurs in an (anti)sequent $S$, we write $\Upsilon[\Upsilon']^{pre}$ (resp.~$\Upsilon[\Upsilon']^{suc}$) to indicate that $\Upsilon'$ is a substructure of $\Upsilon$ and occurs in precedent (respectively, succedent) position, that is, $S$ is display-equivalent to some (anti)sequent $S'$ (notation: $S\equiv S'$)  in which $\Upsilon$ is the sole structure on the left-hand (respectively, right-hand) side of  $\vdash$ or of $\nvdash$. The symbol $\Upsilon[\Upsilon''/\Upsilon']$ denotes the substitution of $\Upsilon''$ for the substructure $\Upsilon'$ within $\Upsilon$.

\noindent We preliminarily  recall that
%\marginnote{\textcolor{red}{ADD:vedete che non va in questa tabella che spreca un sacco di spazio} {\color{blue}AS: potrebbe essere perche' align dentro tabular non e' il massimo, e' l'unico modo che mi e' venuto in mente per fare una tabella con equazioni etichettate per i riferimenti nella dim. Ve ne viene in mente un altro?} GG: Ho provato vari tricks, ma niente... se vogliamo allineare, organizzare tutto su due colonne e usare label, non se ne esce fuori mi pare. Quindi, ho barato e ho usato uno spazio veticale negative. A second di come va quando abbiamo terminate il paper, potrebbe essere necessario riaggustare lo spazio verticale e ci sta che in fase di stampa ce lo cancellino.}
%\begin{fact}\label{fact:DisplayFourCases}
if $\Pi\vdash\Sigma\equiv \Pi'\vdash\Sigma'$  then one of the following cases occurs:
\vspace{-0.7cm}
\begin{center}
\begin{tabular}{p{5cm}p{5.8cm}}
      {\begin{align}
               \Pi'[\Pi] \ \ &\mathrm{and} \ \ \Sigma[\Sigma'] \label{eq:1} \\ 
               \Pi[\Pi'] \ \ &\mathrm{and} \ \ \Sigma'[\Sigma] \label{eq:2} \end{align}}  & 
               {\reqnomode \begin{align}
               \Sigma'[\Pi] \ \ &\mathrm{and} \ \ \Sigma[\Pi'] \label{eq:3} \\
               \Pi[\Sigma'] \ \ &\mathrm{and} \ \ \Pi'[\Sigma] \label{eq:4} \end{align}}
\end{tabular}
\end{center}
\vspace{-0.7cm}

\begin{theorem}\label{thm:soundness}
    The rules of $\mathrm{D.LE}^r$ are sound w.r.t.~the class of complete $\mathcal{L}$-algebras of corresponding signature.
\end{theorem}

\begin{proof}
For axiomatic rules, it suffices to consider an interpretation $v$ such that $v(p)=\top$ and $v(q)=\bot$. For display postulates, reasoning by contraposition (both top-down and bottom-up) is sufficient. For $0$-ary introduction rules for $f$ and $g$ it is enough to take $v$ such that $v(g(\overline{\psi})) = \top$ and $v(f(\overline{\psi})) = \bot$ for all input $\overline{\psi}$.  The soundness of $f_L$ and $g_R$ follows contrapositively from the fact that the following rules are derivable in D.LE:
\begin{center}
    \AXC{$f(\varphi_{1},\ldots,\varphi_{n_{f}})\vdash\Sigma$ }
    %\LL{\fns$f_{L}'$}
    \UIC{$\FH(\varphi_{1},\ldots,\varphi_{n_{f}})\vdash\Sigma$}
    \DP\qquad
    \AXC{$\Pi\vdash g(\varphi_{1},\ldots,\varphi_{n_{g}})$}
    %\RL{\fns$g_{R}'$}
    \UIC{$\Pi\vdash \GC(\varphi_{1},\ldots,\varphi_{n_{g}})$}
    \DP
\end{center}
A similar argument proves the soundness of  $\top_L$ and $\bot_R$. As for the soundness of  $\wedge_{R_i}$ and $\vee_{L_i}$, it is enough to notice that for all elements $a,b,c$ of an arbitrary lattice, $a\nleq b$ implies both $a \nleq b \wedge c$ and $a \vee c \nleq b$.
Now, to prove the soundness of $\FH\ABOT$:
    \begin{center}
        \item
        \AXC{$(\Upsilon_i \nvdash^{\epsilon_f(i)} \ABOT^{\epsilon_f(i)}\mid 1\leq i \leq n_f)$}
            \RL{\fns$\FH\ABOT$}
        \UIC{$\FH(\overline{\Upsilon}) \nvdash \ABOT$}
        \DP
    \end{center}
    we argue contrapositively, by induction on the complexity of $\FH(\overline{\Upsilon}) \vdash \ABOT$, that if $\FH(\overline{\Upsilon}) \vdash \ABOT$ is cut-free derivable, then $\Upsilon_i \vdash^{\epsilon_f(i)} \ABOT^{\epsilon_f(i)}$ is derivable for some $i$. Since $\FH(\overline{\Upsilon}) \vdash \ABOT$ is not display-equivalent to any initial sequent, it  must either result from a logical rule or from  $\top_W$ or $\bot_W$, modulo application of display postulates. We analyze each case. %To prove the soundness of $\FH\ABOT$, we show contrapositively that if $\FH(\overline{\Upsilon}) \vdash \ABOT$ is derivable, then so is $\Upsilon_i \vdash^{\epsilon_f(i)} \ABOT^{\epsilon_f(i)}$ for some $i$. %This claim is proved by induction on the number of logical connectives and atomic variables. 
    %Since $\FH(\overline{\Upsilon}) \vdash \ABOT$ is cut-free derivable and it is not display-equivalent to any initial sequent, it must have been obtained via an application of a logical rule or one of the rules $\top_W$, $\bot_W$, modulo an arbitrary number of display postulates. We examine all the possible cases. 
   
    Case $\wedge_{L_i}$. Assume that $\FH(\overline{\Upsilon}) \vdash \ABOT$ contains a conjunction $A_1 \wedge A_2$ in precedent position, and suppose $A_1\wedge A_2$ is the principal formula in the last (non-display) rule applied in the derivation of $\FH(\overline{\Upsilon}) \vdash \ABOT$. For some structure $\Sigma[\ABOT]$, we have:
        \begin{center}
            \AXC{$A_i \vdash \Sigma[\ABOT]$}
            \LL{\fns$\wedge_{L_i}$}
            \UIC{$A_1\wedge A_2 \vdash \Sigma[\ABOT]$}
            \dashedLine
            %\RL{\fns$\delta$}
            \UIC{$\FH(\overline{\Upsilon})[A_1\wedge A_2]^{pre} \vdash \ABOT$}
            \DP
        \end{center}
        where, above and throughout this section, the dashed line stands for a finite number of applications of display postulates and $A_i \vdash \Sigma[\ABOT] \equiv \FH(\overline{\Upsilon})[A_i/A_1\wedge A_2]^{suc} \vdash \ABOT$ is derivable. By inductive hypothesis, $\Upsilon_j[A_i/A_1\wedge A_2]^{pre} \vdash^{\epsilon_f(j)} \ABOT^{\epsilon_f(j)}$ is derivable for some $j$. If $A_1\wedge A_2$ is not present in $\Upsilon_j$, then $\Upsilon_j[A_i/A_1\wedge A_2] = \Upsilon_j$ and there is nothing else to prove. Otherwise, for some structure $\Sigma'$ we have
        \begin{center}
            \AXC{$\Upsilon_j[A_i/A_1\wedge A_2]^{pre} \vdash^{\epsilon_f(j)} \ABOT^{\epsilon_f(j)}$}
            \dashedLine
            %\RL{\fns$\delta$}
            \UIC{$A_i \vdash \Sigma'[\ABOT^{\epsilon_f(j)}]$}
            \RL{\fns$\wedge_{L_i}$}
            \UIC{$A_1 \wedge A_2 \vdash \Sigma'[\ABOT^{\epsilon_f(j)}]$}
            \dashedLine
            %\RL{\fns$\delta$}
            \UIC{$\Upsilon_j[A_1\wedge A_2/A_1\wedge A_2]^{pre} \vdash^{\epsilon_f(j)} \ABOT^{\epsilon_f(j)}$}
            \DP
        \end{center}
        Since $\Upsilon_j[A_1\wedge A_2/A_1\wedge A_2]^{pre} = \Upsilon_j$, we are done.
        
        Cases $\vee_{R_i}$, $f_L$, $g_R$, $\top_L$, and $\bot_R$ are analogous to the previous case and therefore omitted. In all of them except $\vee_R$, we apply the inductive hypothesis relying on the fact that the complexity decreases whenever a logical connective becomes structural.
        
        Case $\wedge_R$. Assume the structure $\FH(\overline{\Upsilon}) \vdash \ABOT$ contains a conjunction $A \wedge B$ in succedent position, and $A\wedge B$ is the principal formula in the last logical rule (modulo display postulates) applied in the derivation of $\FH(\overline{\Upsilon}) \vdash \ABOT$. For some structure $\Pi$, we have
        \begin{center}
            \AXC{$\Pi \vdash A$}
            \AXC{$\Pi \vdash B$}
            \RL{\fns$\wedge_R$}
            \BIC{$\Pi \vdash A \wedge B$}
            \dashedLine
            %\RL{\fns$\delta$}
            \UIC{$\FH(\overline{\Upsilon})[A\wedge B]^{suc} \vdash \ABOT$}
            \DP
        \end{center}
        where $\Pi \vdash A \equiv \FH(\overline{\Upsilon})[A/A\wedge B]^{suc} \vdash \ABOT$ and $\Pi \vdash B \equiv \FH(\overline{\Upsilon})[B/A\wedge B]^{suc} \vdash \ABOT$ are derivable. We apply the inductive hypothesis twice to obtain $\Upsilon_i[A/A\wedge B]^{suc} \vdash^{\epsilon_f(i)} \ABOT^{\epsilon_f(i)}$ and $\Upsilon_j[B/A\wedge B]^{suc} \vdash^{\epsilon_f(j)} \ABOT^{\epsilon_f(j)}$ for some $i$ and $j$. The only nontrivial case is when $i=j$ and $\Upsilon_i$ contains the aforementioned instance of $A\wedge B$. In this  scenario, for some structure $\Sigma'$,
        \begin{center}
            \AXC{$\Upsilon_i[A/A\wedge B]^{suc} \vdash^{\epsilon_f(i)} \ABOT^{\epsilon_f(i)}$}
            \dashedLine
            %\RL{\fns$\delta$}
            \UIC{$A \vdash \Sigma'$}
            \AXC{$\Upsilon_i[B/A\wedge B]^{suc} \vdash^{\epsilon_f(i)} \ABOT^{\epsilon_f(i)}$}
            \dashedLine
            %\RL{\fns$\delta$}
            \UIC{$B\vdash \Sigma'$}
            \RL{\fns$\wedge_R$}
            \BIC{$A \wedge B \vdash \Sigma'$}
            \dashedLine
            %\RL{\fns$\delta$}
            \UIC{$\Upsilon_i[A \wedge B/A\wedge B]^{suc} \nvdash^{\epsilon_f(i)} \ABOT^{\epsilon_f(i)}$}
            \DP
        \end{center}
      %  Since $\Upsilon_i[A\wedge B/A\wedge B]^{suc} = \Upsilon_i$, we are done. 
      which proves the statement. The case $\vee_L$ is analogous to $\wedge_R$ and is therefore omitted.
        
        Cases $f_R$ and $g_L$ are excluded by virtue of the side condition of $\FH \ABOT$.
        
        Case $\bot_W$. Assume that for some structures $\Pi$ and $\Sigma$ the sequent $\FH(\overline{\Upsilon}) \vdash \ABOT$ is derived as follows, where $\Pi \vdash \ABOT$ is a derivable sequent.
        \begin{center}
            \AXC{$\Pi \vdash \ABOT$}
            \RL{$\bot_W$}
            \UIC{$\Pi \vdash \Sigma$}
            \dashedLine
            %\RL{$\delta$}
            \UIC{$\FH(\overline{\Upsilon})\vdash \ABOT$}
            \DP
        \end{center}
\noindent 
We need to consider four cases. In Case \ref{eq:1}, $\Sigma[\ABOT]$ and $\FH(\overline{\Upsilon})[\Pi]$, assume that $\Pi$ is different from $\FH(\overline{\Upsilon})$ (otherwise $\bot_W$ is applied vacuously), i.e.~$\Upsilon_i[\Pi]$ for some $i$. Hence, $\Upsilon_i \vdash^{\varepsilon_f(i)} \ABOT^{\varepsilon_f(i)} \equiv \Pi \vdash \Sigma'$ for some $\Sigma'$, and the following derivation establishes our thesis:
        \begin{center}
            \AXC{$\Pi \vdash \ABOT$}
            \RL{$\bot_W$}
            \UIC{$\Pi \vdash \Sigma'$}
            \dashedLine
            %\RL{$\delta$}
            \UIC{$\Upsilon_i \vdash^{\varepsilon_f(i)} \ABOT^{\varepsilon_f(i)}$}
            \DP
        \end{center}
Case \ref{eq:2} where $\ABOT[\Sigma]$ and $\Pi[\FH(\overline{\Upsilon})]$ is trivial, as the premise coincides with the endsequent. Case \ref{eq:3} in which $\ABOT[\Pi]$ and $\Sigma[\FH(\overline{\Upsilon})]$ cannot arise: $\Pi$ ($\Sigma$) is in precedent (resp., succedent) position and display rules preserve structure positions. Finally, as to Case \ref{eq:4}, if $\Pi[\ABOT]$ and $\FH(\overline{\Upsilon})[\Sigma]$, i.e.~$\Pi \vdash \ABOT \equiv \FH(\overline{\Upsilon})[\ABOT/\Sigma]^{suc} \vdash \ABOT$ and $\Upsilon_j[\Sigma]$ for some $j$, then $\Sigma$ must be different from an instance of $\ABOT$ (if not, rule $\bot_W$ would have been applied vacuously), hence $\Sigma$ has a higher complexity than $\ABOT$. Hence, the inductive hypothesis yields $\Upsilon_i[\ABOT/\Sigma]^{suc} \vdash^{\varepsilon_f(i)} \ABOT^{\varepsilon_f(i)}$ for some $i$. If $i\neq j$ then $\Upsilon_i[\Sigma]$ does not hold, so the substitution is vacuous and there is nothing else to prove; if $i = j$, then for some  $\Pi'\in \mathsf{Str}_{\mathcal{F}}$,
        \begin{center}
            \AXC{$\Upsilon_i[\ABOT/\Sigma]^{suc} \vdash^{\varepsilon_f(i)} \ABOT^{\varepsilon_f(i)}$}
            \dashedLine
            %\RL{$\delta$}
            \UIC{$\Pi' \vdash \ABOT$}
            \RL{$\bot_W$}
            \UIC{$\Pi' \vdash \Sigma$}
            \dashedLine
            %\RL{$\delta$}
            \UIC{$\Upsilon_i[\Sigma/\Sigma]^{suc} \vdash^{\varepsilon_f(i)} \ABOT^{\varepsilon_f(i)}$}
            \DP
        \end{center}
        where the active instance of $\ABOT$  in the rule $\bot_W$ is the instance of $\ABOT$ that substitutes $\Sigma$ in the previous proof section.

\noindent Case $\top_W$ is similar to the previous case and is therefore omitted.

The proofs of the soundness of $\FH p$ and $f^\neq_R$  closely mirror that of $\FH \ABOT$, except in the $\bot_W$ case with $\Pi[\FH(\overline{\Upsilon})]$, illustrated below. %case. The only exception is in the $\bot_W$ case with $\Pi[\FH(\overline{\Upsilon})]$, where we are in the situation illustrated below.

%As for rules $\FH p$ and $f^\neq_R$, their proof of soundness proceeds virtually identically to the proof for the $\FH \ABOT$ case. The only exception is when dealing with the $\bot_W$ case and $\FH(\overline{\Upsilon}) \triangleleft X$, meaning that we are in the situation illustrated below
\begin{center}
    \begin{tabular}{cc}
        \AXC{$\FH(\overline{\Upsilon}) \vdash \ABOT$}
        \LL{$\bot_W$}
         \UIC{$\FH(\overline{\Upsilon}) \vdash p$}
         \DP
         & 
         \AXC{$\FH(\overline{\Upsilon}) \vdash \ABOT$}
        \RL{$\bot_W$}
         \UIC{$\FH(\overline{\Upsilon}) \vdash f'(\overline{\varphi})$}
         \DP
    \end{tabular}
\end{center}
In both cases, we rely on the soundness of $\FH \ABOT$. To prove the soundness of $f_R$ contrapositively, we must show that if $\FH(\overline{\Upsilon}) \vdash f(\overline{\varphi})$ is derivable, then either $\Upsilon_i \vdash^{\varepsilon_f(i)} \ABOT^{\varepsilon_f(i)}$ holds for some $i$, or $\Upsilon_j \vdash^{\varepsilon_f(i)} \varphi_j$ for all $j$. The proof parallels that for $\FH p$ and $f^\neq_R$, with the added case where the last rule in the derivation of $\FH(\overline{\Upsilon}) \vdash f(\overline{\varphi})$ is
%
%In both cases, we can leverage the soundness of $\FH\ABOT$ to complete the proof. To prove contrapositively the soundness of the rule $f_R$, we need to show that if $\FH(\overline{\Upsilon})\vdash f(\overline{\varphi})$ is derivable either one of $\Upsilon_i \vdash^{\varepsilon_f(i)}\ABOT^{\varepsilon_f(i)}$ is derivable for some $i$ or $\Upsilon_j \vdash^{\varepsilon_f(i)} \varphi_j$ is derivable for all $j$. The proof proceeds identically to the proof for the cases $\FH p$ and $f^\neq_R$, except that there is now a new case to consider, when the last logical rule applied in the derivation of $\FH(\overline{\Upsilon})\vdash f(\overline{\varphi})$ is
\begin{center}
    \AxiomC{$\Big(\Upsilon_i \fCenter^{\!\!\epsilon_{f}(i)}\, \varphi_i  \mid 1\leq i\leq n_f\Big)$}
						\RL{\fns$f_R$}
						\UI$\FH\, (\overline{\Upsilon})\fCenter f(\overline{\varphi})$
						\DP
\end{center}
If that is the case, then clearly $\Upsilon_j \vdash^{\varepsilon_f(i)} \varphi_j$ is derivable for all $j$, as desired.
The proofs of the soundness of  $\AATOP\GC$, $p\GC$, $g^\neq_L$, $g_L$, and $\FH\GC$ are analogous to the cases above.

Finally, a similar argument proves the soundness of the rules $\wedge_L$ and $\vee_R$, but with fewer cases to consider, due to their side conditions: if $\varphi \wedge \psi \vdash \Sigma'$ (resp.~$\Pi' \vdash \varphi \vee \psi$) is derivable, then the last non-display rule applied cannot be $\vee_L$ or $\wedge_R$.
\end{proof}

\begin{theorem}\label{th:refutationcomplete}
    If $\Pi \vdash \Sigma$ is not derivable in $\mathrm{D.LE}$ and does not contain residuals, then $ \Pi \nvdash \Sigma$ is derivable in $\mathrm{D.LE}^r$.
\end{theorem}
\begin{proof}
    By induction on the complexity of the sequent $\Pi \vdash \Sigma$, considering all  cases.

    If $\Pi \coloneqq \bot$ or $\Sigma \coloneqq \top$, the conclusion holds vacuously. 
    
    If $\Pi \vdash \Sigma\equiv \Pi' \vdash A \wedge B$, either $\Pi' \vdash A$ or $\Pi' \vdash B$ is not derivable. If (say) $\Pi' \vdash A \equiv (\Pi \vdash \Sigma)[A/A\wedge B]^{suc}$ is not derivable, by inductive hypothesis we get $\Pi \nvdash \Sigma[A/A\wedge B]^{suc}$. Via display postulates we obtain $\Pi' \nvdash A$, and thus $\Pi' \nvdash A \wedge B$. Hence, we apply the display postulates to get $\Pi \nvdash \Sigma$. The case in which $\Pi \vdash \Sigma\equiv A \vee B \vdash \Sigma'$ is analogous. 

    If $\Pi\vdash\Sigma\equiv f(\overline{\varphi})\vdash\Sigma'$, then $\FH(\overline{\varphi})\vdash\Sigma' \equiv (\Pi\vdash \Sigma)[\FH(\overline{\varphi})/f(\overline{\varphi})]^{pre}$ is not derivable, and by inductive hypothesis we get $\FH(\overline{\varphi})\nvdash\Sigma'$. We apply $f_L$ and display postulates to get $\Pi\nvdash\Sigma$. The cases in which $\Pi\vdash\Sigma$ is display-equivalent either to $\Pi'\vdash g(\overline{\varphi})$, $\Pi'\vdash\bot$, or $\top \vdash \Sigma'$, are analogous.

    Given the previous cases, we assume that $\Pi \vdash \Sigma$ is non-branching and displays no $f \in \mathcal{F}$ or $\top$ (resp.~$g \in \mathcal{G}$ or $\bot$) in precedent (resp.~succedent) position. Let us treat the case of $\Sigma \coloneqq \ABOT$ (the case of $\Pi\coloneqq\ATOP$ is analogous) and consider the possible subcases.
   %From the previous cases, from the remainder of the proof we can assume that $\Pi \vdash \Sigma$ is not branching and there are no $f \in \mathcal{F}$ or $\top$ (resp.~$g \in \mathcal{G}$ or $\bot$) connectives displayable on the left (resp.~right). Let us now assume $\Sigma \coloneqq \ABOT$ and distinguish all the remaining cases.
   % \begin{itemize}
       % \item 

        \noindent 
    -- If $\Pi \coloneqq \AATOP$, $\Pi \coloneqq p$, and $\Pi \coloneqq g(\overline{\varphi})$, then $\Pi \nvdash \Sigma$ is derivable via $\mathrm{Ax}1$, $\mathrm{Ax}2$, and $g \ABOT$, respectively. 
    
       % \item
       \noindent 
     --  $\Pi \coloneqq \FH(\overline{\Upsilon})$. By assumption, $\hat{f}(\Upsilon_1,\ldots,\Upsilon_{n_f})\vdash \ABOT$ is not derivable. Let $\upsilon_i\in \mathcal{L}$ be the operational counterpart  of $\Upsilon_i$  for every $i$ in $\{1,\ldots,n_f\}$ ($\upsilon_i$ exists since $\Upsilon_i$ does not contain residuals). Hence, $\FH(\overline{\Upsilon})\vdash f(\overline{\upsilon})$ is derivable and, for every $\Sigma'$, if $\Upsilon_i\vdash\Sigma'$ is derivable, so is $\upsilon_i\vdash\Sigma'$. We claim that $\Upsilon_i \vdash^{\varepsilon_f(i)}\check{\bot}^{\varepsilon_f(i)}$ is not derivable, for any $1\leq i\leq n_{f}$. Assume the contrary, and, w.l.o.g., that $i=1$ and $\varepsilon_f(1)=1$ (the other cases being analogous). Then we can derive $\hat{f}(\Upsilon_1,\ldots,\Upsilon_{n_f})\vdash \ABOT$ as follows, against our assumption.

\begin{center}
    \AXC{$\vdots$}
    \noLine
    \UIC{$\hat{f}(\Upsilon_1,\ldots,\Upsilon_{n_f})\vdash f(\upsilon_1,\ldots,\upsilon_{n_f})$}

    \AXC{$\Upsilon_1\vdash\ABOT$}
    \LL{$\bot_R$}
    \UIC{$\Upsilon_1 \vdash \bot$}
    \AXC{}
    \RL{$\bot_L$}
    \UIC{$\bot\vdash\ABOT$}
    \RL{$\bot_W$}
    \UIC{$\bot\vdash \check{f}_1^\sharp(\ABOT,\upsilon_2,\ldots,\upsilon_{n_f})$}
    \RL{\fns Cut}
    \BIC{$\Upsilon_1 \vdash\check{f}_1^\sharp(\ABOT,\upsilon_2,\ldots,\upsilon_{n_f})$}
    \noLine
    \UIC{$\vdots$}
    \noLine
    \UIC{$\upsilon_1 \vdash\check{f}_1^\sharp(\ABOT,\upsilon_2,\ldots,\upsilon_{n_f})$}
    \LL{\fns $\FH \nvdash \FCS_1$}
    \UIC{$\hat{f}(\upsilon_1,\ldots,\upsilon_{n_f})\vdash\ABOT$}
    \RL{$f_L$}
    \UIC{$f(\upsilon_1,\ldots,\upsilon_{n_f})\vdash\ABOT$}
    \RL{\fns Cut}
    \BIC{$\hat{f}(\Upsilon_1,\ldots,\Upsilon_{n_f})\vdash \ABOT$}
    \DP
\end{center}

 Hence, no sequent $\Upsilon_i \vdash^{\varepsilon_f(i)} \check{\bot}^{\varepsilon_f(i)}$ is derivable:  then, by inductive hypothesis, $\Upsilon_i \nvdash^{\varepsilon_f(i)} \check{\bot}^{\varepsilon_f(i)}$, for any $1\leq i\leq n_{f}$. Finally, we apply  $\bot_{f}$ to derive $\Pi \nvdash \Sigma$.
%\item 

\noindent 
-- $\Pi \coloneqq A \wedge B$. If $A \wedge B \vdash \ABOT$ is not derivable, then neither of $A \vdash \ABOT$ and $B \vdash \ABOT$ are. Then by inductive hypothesis,  $A \nvdash \ABOT$ and $B \nvdash \ABOT$, and an application of $\wedge_L$ yields $A \wedge B \nvdash \ABOT$.
 %The sequents $A \vdash \ABOT$ and $B \vdash \ABOT$ are not derivable, otherwise, $A \wedge B \vdash \ABOT$ would be derivable. By the inductive hypothesis, we deduce that $A \nvdash \ABOT$ and $B \nvdash \ABOT$ are derivable in our refutational calculus, and since $\ABOT$ is not branching, applying $\wedge_L$ we get $A \wedge B \nvdash \ABOT$.
   % \end{itemize}

    Next, we assume that $\Sigma \coloneqq p$ (the case of $\Pi \coloneqq p$ is similar) and distinguish such subcases.
  %  \begin{itemize}
       % \item 

       \noindent
     --  $\Pi \coloneqq q \neq p$ and $\Pi \coloneqq g(\overline{\varphi})$. The antisequent $\Pi \nvdash \Sigma$ is derivable via $\mathrm{Ax}4$ and $g p$, respectively.
     
        \noindent
     -- $\Pi \coloneqq \FH(\overline{\Upsilon})$ and $\Pi \coloneqq A \wedge B$. Similar to the subcases $\FH(\overline{\Upsilon}) \vdash \ABOT$ and $A \wedge B \vdash \ABOT$   above.
    %\end{itemize}
    
    \smallskip
Let us now assume that $\Sigma \coloneqq f(\overline{\varphi})$ (the case in which $\Pi \coloneqq g(\overline{\varphi})$ is analogous).
   % \begin{itemize}
     %   \item 
     
     \noindent
     --    $\Pi \coloneqq g(\overline{\varphi'})$. The antisequent $\Pi \nvdash \Sigma$ is derivable via the initial rule $g f$.
     
         \noindent
     -- $\Pi \coloneqq \hat{f'}(\overline{\Upsilon})$ with $f' \neq f$. Similar to the subcase $\FH(\Upsilon) \vdash \ABOT$ analyzed above.
     
        \noindent
     -- $\Pi \coloneqq \FH(\overline{\Upsilon})$. As in the previous $\hat{f}(\Upsilon_1,\ldots,\Upsilon_{n_f})\vdash \ABOT$ case, if $\Pi\vdash\Sigma$ is not derivable, then $\Upsilon_i \vdash^{\varepsilon_f(i)} \check{\bot}^{\varepsilon_f(i)}$ is not derivable for all $i$ in $\{1,\ldots,n_f\}$. Moreover, for some $j$, the sequent $\Upsilon_j \vdash^{\varepsilon_f(j)} \varphi_j$ is also not derivable, for otherwise one would get $\hat{f}(\Upsilon_1,\ldots,\Upsilon_{n_f})\vdash f(\varphi_1, \ldots, \varphi_{n_{f}})$ as shown below:
        \begin{center}
            \AxiomC{$\Big(\Upsilon_i \fCenter^{\!\!\epsilon_{f}(i)}\, \varphi_i  \mid 1\leq i\leq n_f\Big)$}
						\RL{\fns$f_R$}
						\UI$\FH\, (\overline{\Upsilon})\fCenter f(\overline{\varphi})$
						\DP
        \end{center}
        By the inductive hypothesis, we conclude  $\Upsilon_i \nvdash^{\varepsilon_f(i)} \check{\bot}^{\varepsilon_f(i)}$ and $\Upsilon_j \nvdash^{\varepsilon_f(j)} \varphi_j$ for all $j$ and some $i$, yielding $\Pi \nvdash \Sigma$ via $f_R$.
        
            \noindent
     -- $\Pi \coloneqq A \wedge B$. Similar to the subcase $A\wedge B \vdash \ABOT$ analyzed above.
     
     \smallskip
    %\end{itemize}
    Now, suppose that $\Sigma \coloneqq A \vee B$ (the case in which $\Pi\coloneqq A\wedge B$ is analogous).
    
        \noindent
     -- $\Pi \coloneqq \FH(\overline{\Upsilon})[C_i\wedge D_i]^{pre}_{i\in I}[C'_j\vee D'_j]^{suc}_{j\in J}$. None of the sequents 
        \[\FH(\overline{\Upsilon}) \vdash A,\quad \FH(\overline{\Upsilon}) \vdash B,\quad \FH(\overline{\Upsilon})[C_i]^{pre}_i \vdash A \vee B,\quad \FH(\overline{\Upsilon})[D_i]^{pre}_i \vdash A \vee B,\]
        \[\FH(\overline{\Upsilon})[C'_j]^{suc}_j \vdash A \vee B,\quad \FH(\overline{\Upsilon})[D'_j]^{suc}_j \vdash A \vee B,\]
        are derivable for all $i \in I$ and $j \in J$, for otherwise, one would be able to derive $\FH(\overline{\Upsilon}) \vdash A \vee B$ applying the appropriate lattice rules modulo zero or more display postulates. By inductive hypothesis, all the sequents
        \[\FH(\overline{\Upsilon}) \nvdash A,\quad \FH(\overline{\Upsilon}) \nvdash B,\quad \FH(\overline{\Upsilon})[C_i]^{pre}_i \nvdash A \vee B,\quad \FH(\overline{\Upsilon})[D_i]^{pre}_i \nvdash A \vee B,\]
        \[\FH(\overline{\Upsilon})[C'_j]^{suc}_j \nvdash A \vee B,\quad \FH(\overline{\Upsilon})[D'_j]^{suc}_j\nvdash A \vee B,\]
        are derivable for all $i\in I$ and $j \in J$, thus we get $\FH(\overline{\Upsilon}) \nvdash A \vee B$ via the rule $\vee_R$.
        
          \noindent
     -- $\Pi \coloneqq C \wedge D$. This subcase is similar to the previous one.
    %\end{itemize}
    
    \noindent Finally, if $\Pi \vdash \Sigma$ is $\FH(\Upsilon_{1},\ldots,\Upsilon_{n_f})\vdash\GC\, (\Upsilon'_1,\ldots, \Upsilon'_{n_g})$, we proceed as in the case of $\FH(\overline{\Upsilon}) \vdash \ABOT$ to get the conclusion.  
\end{proof}
\begin{corollary}\label{cor:completeness}

For any language $\mathcal{L}_\mathrm{LE}$,   the refutation calculus $\mathrm{D.LE^{r}}$ is complete, and the basic normal $\mathcal{L}_\mathrm{LE}$-logic $\mathbf{L}_\mathrm{LE}$  is decidable.
\end{corollary}
\begin{proof}
  The first part is immediate from  Theorem \ref{th:refutationcomplete} and the completeness of D.LE. As to the second part, for any pair of structures $\Pi,\Sigma$ without residuals, completeness of D.LE and Theorem \ref{th:refutationcomplete} ensure that either $\Pi\vdash\Sigma$ is derivable in $\mathrm{D.LE}$ or $\Pi\nvdash\Sigma$ in $\mathrm{D.LE^{r}}$. The thesis follows from the observation that $\mathbf{L}_{\mathrm{LE}}$-sequents do not contain residuals. %For any pair of multisets $\Pi,\Sigma$ containing either $\FH$ or $\GC$, there exist $\Pi',\Sigma'$ without $\FH$ or $\GC$ such that $\Pi\vdash\Sigma\equiv \Pi'\vdash\Sigma'$. Since $\Pi\vdash\Sigma$ is derivable in $\mathrm{D.LE}$ exactly when $\Pi'\vdash\Sigma'$ is derivable in $\mathrm{D.LE}$, the conclusion follows. 
\end{proof}

\section{Tableaux for LE-logics}\label{sec:tableau}

In this section, we define tableaux rules for LE-logics by considering contrapositive versions of $\mathrm{D.LE}^r$ rules. A refutation in the tableaux presentation of the calculus is a rooted graph where every node is labelled with a sequent, the root is labelled with the sequent to be refuted, and each other node is introduced using a rule of the calculus.\footnote{We opted to directly manipulate sequents rather than (labelled) formulas, since this presentation is more compact and facilitates the comparison with the sequent calculus $\mathrm{D.LE}^r$.} As usual in the case of tableaux, if two nodes are introduced in different branches in the conclusion of a rule, then the reading is disjunctive; if two nodes are introduced on the same branch in the conclusion of a rule, then the reading is conjunctive.\footnote{Notice that the conclusion of each display rule depends on the order type of the $i$-th (resp.~$j$-th) coordinate, therefore each rule is a shorthand for two rules.} 
\begin{itemize}
\item Display rules$^{\star}$ for $f\in\mathcal{F}$ and $g\in\mathcal{G}$:
\end{itemize}
\begin{center}
\footnotesize{
\begin{tabular}{ccc}
    \begin{forest}
[ {$\FH (\Upsilon_1,\ldots,\Upsilon_i,\ldots \Upsilon_{n_{f}}) \vdash \Sigma$} 
[ {$\Upsilon_{i} \vdash^{\epsilon_{f}(i)} \mathring{f}^{\,\sharp}_{i}\, (\Upsilon_1,\ldots,\Sigma,\ldots, \Upsilon_{n_{f}})$}]
]
\end{forest}
& \ \ \ \ \ \ \ \ \ \ \ \ \ &
 \begin{forest}
[ {$\Upsilon_{i} \vdash^{\epsilon_{f}(i)} \mathring{f}^{\,\sharp}_{i}\, (\Upsilon_1,\ldots,\Sigma,\ldots, \Upsilon_{n_{f}})$}
[ {$\FH (\Upsilon_1,\ldots,\Upsilon_i,\ldots \Upsilon_{n_{f}}) \vdash \Sigma$} ]
]
\end{forest}
\end{tabular}}
\end{center}

\begin{center}
\footnotesize{
\begin{tabular}{ccc}
\begin{forest}
[ {$\Pi \vdash \GC (\Upsilon_1,\ldots,\Upsilon_j,\ldots \Upsilon_{n_{g}})$} 
[{$\mathring{g}^{\,\flat}_{j}\, (\Upsilon_1,\ldots, \Pi,\ldots \Upsilon_{n_{g}}) \vdash^{\epsilon_{g}(j)} \Upsilon_{j}$}]
]
\end{forest}
& \ \ \ \ \ \ \ \ \ \ \ \ \ &
\begin{forest} 
[{$\mathring{g}^{\,\flat}_{j}\, (\Upsilon_1,\ldots, \Pi,\ldots \Upsilon_{n_{g}}) \vdash^{\epsilon_{g}(j)} \Upsilon_{j}$} [ {$\Pi \vdash \GC (\Upsilon_1,\ldots,\Upsilon_1,\ldots \Upsilon_{n_{g}})$} ]]
\end{forest}
\end{tabular}}
\end{center}

$^{\star}$Notational convention: $1\leq i\leq n_{f}$, $1\leq j\leq n_{g}$. Furthermore, 
$\mathring{f}^{\,\sharp}_{i}$ (resp.~$\mathring{g}^{\,\flat}_{i}$) is $\FC^{\,\sharp}_{i}$ (resp.~$\GH^{\,\flat}$) if $\varepsilon_{f}(i) = 1$, and $\FH^{\,\sharp}_{i}$ (resp.~$\GC^{\,\flat}$) otherwise.

$^{\star}$Side condition: any residuation rule can be applied if the conclusion has not been obtained by some previous rule application.
\begin{itemize}
\item Structural rules$^{*}$ for $f\in\mathcal{F}$ and $g\in\mathcal{G}$:
\end{itemize}
\begin{center}
\footnotesize{
\begin{tabular}{ccc}
\begin{forest}
[ {$\FH\, (\overline{\Upsilon}) \vdash \ABOT$}
[ {$\Upsilon_1 \vdash^{\epsilon_f(1)} \ABOT^{\epsilon_f(1)}$}]
[$\cdots$, no edge]
[ {$\Upsilon_{n_{f}} \vdash^{\epsilon_f(n_{f})} \ABOT^{\epsilon_f(n_{f})}$}]]
\end{forest} 
 & \ \ \ \ \ & 
\begin{forest}
[ {$\AATOP \vdash \GC\, (\overline{\Upsilon'})$}
[ {$\AATOP^{\epsilon_g(1)} \vdash^{\epsilon_g(1)} \Upsilon'_1$}]
[$\cdots$, no edge]
[ {$\AATOP^{\epsilon_g(n_{g})} \vdash^{\epsilon_g(n_{g})} \Upsilon'_{n_{g}}$}
]]
\end{forest}
\end{tabular}}
\end{center}

\begin{center}
\footnotesize{
\begin{tabular}{ccc}
\begin{forest}
[ {$\FH\, (\overline{\Upsilon}) \vdash p$}
[ {$\Upsilon_1 \vdash^{\epsilon_f(1)} \ABOT^{\epsilon_f(1)}$}]
[$\cdots$, no edge]
[ {$\Upsilon_{n_{f}} \vdash^{\epsilon_f(n_{f})} \ABOT^{\epsilon_f(n_{f})}$}]
]
\end{forest} 
 & \ \ \ \ \ &  
\begin{forest}
[ {$p \vdash\GC\, (\overline{\Upsilon'})$}
[ {$\AATOP^{\epsilon_g(1)} \vdash^{\epsilon_g(1)} \Upsilon'_1$}]
[{$\cdots$}, no edge]
[ {$\AATOP^{\epsilon_g(n_{g})} \vdash^{\epsilon_g(n_{g})} \Upsilon'_{n_{g}}$}]]
\end{forest}  
 \\
\end{tabular}}
\end{center}

\begin{center}
    \footnotesize{
    \begin{forest}
[ {$\FH(\overline{\Upsilon}) \vdash \GC\, (\overline{\Upsilon'})$} [{$\Upsilon_1 \vdash^{\epsilon_f(1)} \ABOT^{\epsilon_f(1)}$}]
[{$\cdots$}, no edge]
[{$\Upsilon_{n_{f}} \vdash^{\epsilon_f(n_{f})} \ABOT^{\epsilon_f(n_{f})} $}]
[{$\AATOP^{\epsilon_g(1)}\vdash^{\epsilon_g(1)} \Upsilon'_1$}]
[{$\cdots$}, no edge]
[{$\AATOP^{\epsilon_g(n_{g})} \vdash^{\epsilon_g(n_{g})} \Upsilon'_{n_{g}}$}]
]
\end{forest}
    }
\end{center}
$^{*}$Side condition: in each rule above, the premise should not contain residuals.
\begin{itemize}
\item Logical rules$^{**}$ for $f\in\mathcal{F}$ and $g\in\mathcal{G}$:
\end{itemize}
\begin{center}
\footnotesize{
\begin{tabular}{ccc}
\begin{forest}
[ {$f (\Upsilon_1,\ldots, \Upsilon_{n_{f}}) \vdash \Upsilon$} 
[ {$\FH\, (\Upsilon_1,\ldots, \Upsilon_{n_{f}}) \vdash \Upsilon $} ] 
]
\end{forest}
& \ \ \ \ \ \ \ \ \ \ &
\begin{forest}
[ {$\Upsilon \vdash g (\Upsilon_1,\ldots, \Upsilon_{n_{g}})$} 
[ {$\Upsilon \vdash \GC\, (\Upsilon_1,\ldots, \Upsilon_{n_{g}})$}]
]
\end{forest}
 \\
\end{tabular}}
\end{center}
\begin{center}
\footnotesize{
\begin{forest}
[ {$\FH(\overline{\Upsilon})\vdash f(\overline{\varphi})$} [ {$\Upsilon_1 \vdash^{\epsilon_f(1)} \ABOT^{\epsilon_f(1)}$}]
[ {$\cdots$}]
[$\Upsilon_{n_{f}} \vdash^{\epsilon_f(n_{f})} \ABOT^{\epsilon_f(n_{f})}$]
[ {$\Upsilon_1 \vdash^{\epsilon_{f(1)}} \varphi_1$} [ {$\Upsilon_{n_{f}} \vdash^{\epsilon_{f(n_{f})}} \varphi_{n_{f}}$}, edge={dotted,thick} ]]
]
\end{forest}}
\end{center}

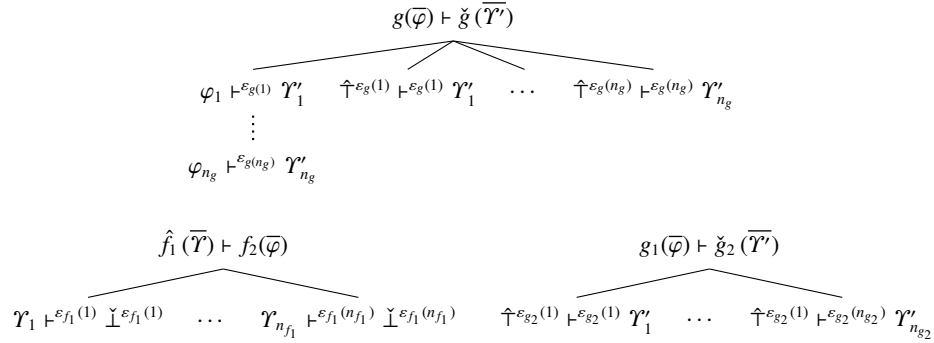
\begin{center}
\footnotesize{
\begin{forest}
[ {$g(\overline{\varphi}) \vdash \GC\, (\overline{\Upsilon'})$}
[ {$\varphi_1 \vdash^{\epsilon_{g(1)}} \Upsilon'_1$} [ {$\varphi_{n_{g}} \vdash^{\epsilon_{g(n_{g})}} \Upsilon'_{n_{g}}$}, edge={dotted,thick} ]]
[{$\AATOP^{\epsilon_g(1)} \vdash^{\epsilon_g(1)} \Upsilon'_1$}]
[ {$\cdots$}]
[ {$\AATOP^{\epsilon_g(n_{g})} \vdash^{\epsilon_g(n_{g})} \Upsilon'_{n_{g}}$} ]
]
\end{forest}}  
\end{center}

\begin{center}
\footnotesize{
\begin{tabular}{ccc}
\begin{forest}
[ {$\FHp\, (\overline{\Upsilon}) \vdash f_2(\overline{\varphi})$} 
[ {$\Upsilon_1 \vdash^{\epsilon_{f_1}(1)} \ABOT^{\epsilon_{f_1}(1)}$}]
[{$\cdots$}, no edge]
[ {$\Upsilon_{n_{f_{1}}} \vdash^{\epsilon_{f_1}(n_{f_{1}})} \ABOT^{\epsilon_{f_1}(n_{f_{1}})}$}] ]
\end{forest}
 & \ \ \  & 
\begin{forest}
[ {$g_1(\overline{\varphi}) \vdash \GC_2\, (\overline{\Upsilon'})$} 
[ {$\AATOP^{\epsilon_{g_2}(1)} \vdash^{\epsilon_{g_2}(1)} \Upsilon'_1$}]
[{$\cdots$}, no edge]
[ {$\AATOP^{\epsilon_{g_2}(1)} \vdash^{\epsilon_{g_2}(n_{g_{2}})} \Upsilon'_{n_{g_2}}$}] ]
\end{forest}
 \\
\end{tabular}}
\end{center}

 %Moreover, $f_{(1)}(\overline{\Upsilon}) = f(\Upsilon_1, \ldots, \Upsilon_{n_{f_{(1)}}})$, $g_{(2)}(\overline{\Upsilon}) = g(\Upsilon_1, \ldots, \Upsilon_{n_{g_{(2)}}})$, $f_{(2)}(\overline{\varphi})=f(\varphi_{1},\ldots,\varphi_{n_{f_{(2)}}})$, $g_{(1)}(\overline{\varphi})=g(\varphi_{1},\ldots,\varphi_{n_{g_{(1)}}})$, $\FH_{(1)}(\overline{\Upsilon}) = f(\Upsilon_1, \ldots, \Upsilon_{n_{f_{(1)}}})$ and $\GC_{(2)}(\overline{\Upsilon'}) = g(\Upsilon'_1, \ldots, \Upsilon'_{n_{g_{(2)}}})$.  

$^{**}$Side condition: $f_{1}\not=f_{2}$ and $g_{1}\not=g_{2}$, and the premises of all the other rules above should not contain residuals.
\begin{itemize}
\item Rules$^{***}$ for lattice connectives:
\end{itemize}
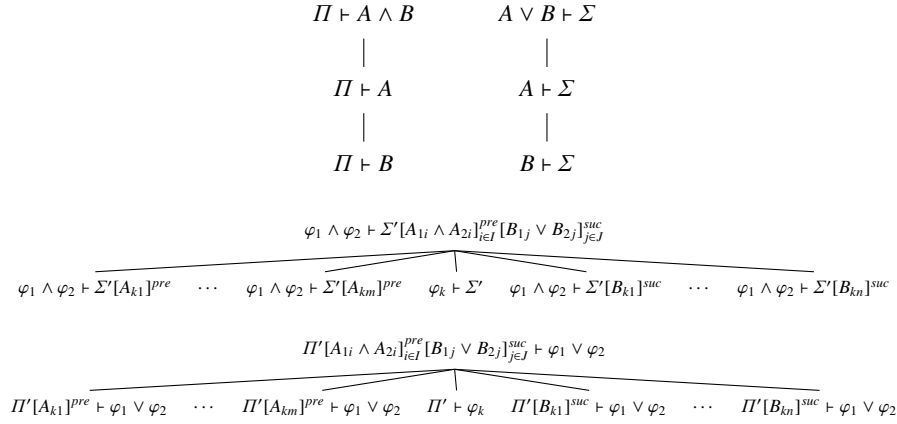
\begin{center}
\small{
\begin{tabular}{ccc}
%\begin{forest}
%[{$\Pi\vdash A\wedge B$}
%[{$\begin{matrix}\Pi\vdash A \\ \Pi\vdash B \end{matrix}$} ]]
%\end{forest}
\begin{forest}
[{$\Pi\vdash A\wedge B$}
[{$\Pi\vdash A$} [{$\Pi\vdash B$}] ]]
\end{forest}
& \ \ \ \ \ \ \ \ \ \ &
%\begin{forest}
%[{$A\vee B\vdash\Sigma$}
%[{$\begin{matrix}A\vdash\Sigma \\ B\vdash\Sigma \end{matrix}$}]]
%\end{forest}
\begin{forest}
[{$A\vee B\vdash\Sigma$}
[{$A\vdash\Sigma$} [{$B\vdash\Sigma$}] ]]
\end{forest} \\
\end{tabular}}
\end{center}

\begin{center}
\scriptsize{
    \begin{forest}
[{$\varphi_{1} \wedge \varphi_{2} \vdash \Sigma'[A_{1i}\wedge A_{2i}]_{i\in I}^{pre}[B_{1j}\vee B_{2j}]_{j\in J}^{suc}$}
    [{$\varphi_{1} \wedge \varphi_{2} \vdash \Sigma'[A_{k1}]^{pre}$}]
    [{$\cdots$}, no edge]
    [{$\varphi_{1} \wedge \varphi_{2} \vdash \Sigma'[A_{km}]^{pre}$}]
    [{$\varphi_{k} \vdash \Sigma'$}]
    %[{$\varphi_{2} \vdash \Sigma'$}]
    [{$\varphi_{1} \wedge \varphi_{2} \vdash \Sigma'[B_{k1}]^{suc}$}]
    [{$\cdots$}, no edge]
    [{$\varphi_{1} \wedge \varphi_{2} \vdash \Sigma'[B_{kn}]^{suc}$}]
]
\end{forest}}
\end{center}

\begin{center}
\scriptsize{
\begin{forest}
[{$\Pi'[A_{1i}\wedge A_{2i}]_{i\in I}^{pre}[B_{1j}\vee B_{2j}]_{j\in J}^{suc} \vdash \varphi_{1}\vee\varphi_{2}$}
    [{$\Pi'[A_{k1}]^{pre}\vdash\varphi_{1} \vee\varphi_{2} $}]
    [{$\cdots$}, no edge]
    [{$\Pi'[A_{km}]^{pre}\vdash\varphi_{1} \vee\varphi_{2} $}]
    [{$\Pi' \vdash \varphi_{k}$}]
    %[{$\Pi' \vdash \varphi_{2}$}]
    [{$\Pi'[B_{k1}]^{suc}\vdash \varphi_{1} \vee \varphi_{2} $}] 
    [{$\cdots$}, no edge]
    [{$\Pi'[B_{kn}]^{suc}\vdash \varphi_{1} \vee \varphi_{2} $}] 
]
\end{forest}}
\end{center}

$^{***}$Notational convention: $I=\{1,\ldots,m\}$ and $J=\{1,\ldots,n\}$. Furthermore, $k$ ranges in $\{1,2\}$ (meaning that the last two rules above create $2m +2n +2$ new branches).

$^{***}$Side condition: $\Pi'$ and $\Sigma'$ are not branching, and $\varphi_{1} \wedge \varphi_{2} \vdash \Sigma'[A_{1i}\wedge A_{2i}]_{i\in I}^{pre}[B_{1j}\vee B_{2j}]_{j\in J}^{suc}$ and $\Pi'[A_{1i}\wedge A_{2i}]_{i\in I}^{pre}[B_{1j}\vee B_{2j}]_{j\in J}^{suc} \vdash \varphi_{1}\vee\varphi_{2}$ do not contain residuals.

A branch of a tableau tree is {\em terminated} if for any sequent $\Pi'\vdash\Sigma'$ belonging to a node along the branch, no tableau rule can be applied to $\Pi'\vdash\Sigma'$ without generating a sequent that already occurs in the branch. A tableau tree is terminated if and only if each of its branches is terminated.

We say that a terminated branch is {\em open} whenever it contains (at least) one sequent of one of the following forms:
%We say that a terminated branch is {\em open} whenever there exists (at least) one sequent $\Pi'\vdash\Sigma'$ belonging to its terminal node which has one of the following forms:
    \[\AATOP \vdash \ABOT \quad p \vdash \ABOT \quad \AATOP \vdash q \quad p \vdash q\]
    \[g(\overline{\psi}) \vdash \ABOT \quad g(\overline{\psi}) \vdash p \quad g(\overline{\psi}) \vdash f(\overline{\varphi}) \quad  p\vdash f(\overline{\psi}) \quad \AATOP\vdash f(\overline{\psi})\] 
A terminated branch is {\em closed} if it is not open. 
%whenever for any sequent $\Pi'\vdash\Sigma'$ belonging to its terminal node, if $\Pi'\vdash\Sigma'$ does not contain structural connectives $\FH$ (resp. $\GC$) with $f\in\mathcal{F}^{\ast}\setminus\mathcal{F}$ (resp. $g\in\mathcal{G}^{\ast}\setminus\mathcal{G}$), then $\Pi'\vdash\Sigma'$ has the form $p\vdash p$.
We say that a terminated tableau tree is closed exactly when it has (at least) one closed branch, and open otherwise.

\begin{example}
\remove{The following is a closed tableau tree for $g(p)\wedge g(q)\vdash g(p\wedge q)$:
\begin{center}
\footnotesize{
    \begin{forest}
        [{$g(p)\wedge g(q)\vdash g(p\wedge q)$}
        [{$g(p)\wedge g(q)\vdash\GC(p\wedge q)$}
        [{$\GH^{\,\flat}(g(p)\wedge g(q))\vdash p\wedge q$}
        [{$\GH^{\,\flat}(g(p)\wedge g(q))\vdash p$}
        [{$\GH^{\,\flat}(g(p)\wedge g(q))\vdash q$}
        [{$g(p)\wedge g(q)\vdash \GC(p)$}
        [{$g(p)\wedge g(q)\vdash \GC(q)$}
        [{$g(p)\vdash\GC(p)$}
        [{$g(p)\vdash\GC(q)$}
        [{$p\vdash p$}
        [{$p\vdash q$}]
        [{$\ATOP\vdash q$}]
        ]
        [{$\ATOP\vdash p$}
        [{$p\vdash q$}]
        [{$\ATOP\vdash q$}]
        ]
        ]
        [{$g(q)\vdash\GC(q)$}
        [{$p\vdash p$}
        [{$q\vdash q $}]
        [{$\ATOP\vdash q$}]
        ]
        [{$\ATOP\vdash p$}
        [{$q\vdash q$}]
        [{$\ATOP\vdash q$}]
        ]
        ]
        ]
        [{$g(q)\vdash\GC(p)$}
        [{$g(p)\vdash\GC(q)$}
        [{$q\vdash p$}
        [{$p\vdash q$}]
        [{$\ATOP\vdash q $}]
        ]
        [{$\ATOP\vdash p$}
        [{$p\vdash q$}]
        [{$\ATOP\vdash q $}]
        ]
        ]
        [{$g(q)\vdash\GC(q)$}
        [{$q\vdash p$}
        [{$q\vdash q$}]
        [{$\ATOP\vdash q$}]
        ]
        [{$\ATOP\vdash p$}
        [{$q\vdash q$}]
        [{$\ATOP\vdash q$}]
        ]
        ]
        ]
        ]
        ]
        ]
        ]
        ]
        ]
        ]
    \end{forest}}
\end{center}}
This is an open tableau tree for $g(p\vee q)\vdash g(p)\vee g(q)$, for some $g\in\mathcal{G}$ with order-type $\langle1\rangle$:
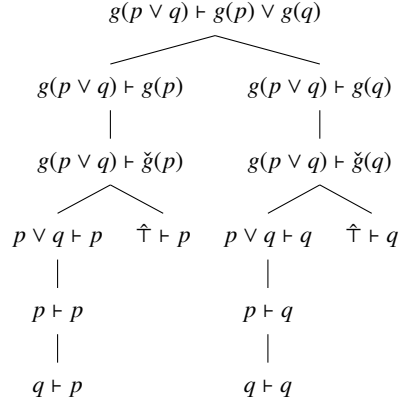
\begin{center}
\small{
     \begin{forest}
        [{$g(p\vee q)\vdash g(p)\vee g(q)$}
        [{$g(p\vee q)\vdash g(p)$}
        [{$g(p\vee q)\vdash \GC(p)$}
        [{$p\vee q\vdash p$}
        [{$p\vdash p$}
        [{$q\vdash p$}]
        ]
        ]
        [{$\ATOP\vdash p$}]
        ]
        ]
        [{$g(p\vee q)\vdash g(q)$}
        [{$g(p\vee q)\vdash \GC(q)$}
        [{$p\vee q\vdash q$}
        [{$p\vdash q$}
        [{$q\vdash q$}]
        ]
        ]
        [{$\ATOP\vdash q$}]
        ]
        ]
        ]
    \end{forest}}
\end{center}
\end{example}

\begin{theorem}
A sequent $\Pi\vdash\Sigma$ is valid iff  some closed tableau tree exists for it. 
\end{theorem}

\begin{proof}
For any tableau rule different from a residuation one, the number of connectives in each conclusion is lower than the number of connectives in the premise. Moreover, each sequent has finitely many display-equivalent  sequents, and the side condition on the applicability of display rules prevents cycles. Hence, for any sequent $\Pi\vdash\Sigma$, there exists (at least) one terminated tableau tree. We leverage Theorem \ref{thm:soundness} and Corollary \ref{cor:completeness} to complete the proof.
\end{proof}

\section{Conclusions}
\label{sec:conclusions}

This work opens several directions for future research in the study of refutation display calculi. Starting from LE-logics appears to be a promising first step, given their general semantic framework.
It  would also be valuable to extend the current platform to  refutation display calculi for {\em distributive} LE-logics. Moreover, a broader perspective might involve addressing {\em axiomatic extensions} of LE-logics, potentially through the use of {\em hybrid} refutation rules that incorporate both sequents and antisequents.

The main technical motivation for employing display calculi stems from the metatheorem in \cite{Belnap}, which establishes sufficient conditions under which a sequent calculus admits Gentzen-style cut elimination. In systems that include both sequents and antisequents, it becomes possible to define anticut rules -- that is, contrapositive analogues of the standard cut rule \cite{anticut}. It would be intriguing to determine whether sufficient conditions can be specified to support a uniform set of transformation steps enabling anticut elimination.

\bibliographystyle{abbrv}
\bibliography{reference}

\end{document}